\documentclass[11pt,letterpaper]{amsart}
\usepackage{url,amsmath,amsthm, amssymb}
\usepackage[margin=1in]{geometry}
\usepackage[pdfencoding=auto, psdextra, colorlinks=true, linkcolor = blue, urlcolor=blue, citecolor=blue,pagebackref,breaklinks]{hyperref}
\usepackage{bookmark}
\usepackage[all]{mathtools}
\usepackage{makecell}
\usepackage{mathrsfs}
\usepackage{mathdots}
\usepackage{amsxtra,amscd}
\usepackage{amsrefs}
\usepackage{hyperref}
\usepackage{enumerate, enumitem}
\usepackage{arydshln}
\usepackage{multirow}
\usepackage{bbm} 
\usepackage[all,arc]{xy}

 \usepackage{amssymb,amsrefs}
\usepackage{amsmath}
\usepackage{amscd,latexsym}
\usepackage{bookmark}
 \usepackage{amssymb,amsfonts,bm}
\usepackage[all,arc]{xy}
\usepackage{enumerate}
\usepackage{mathrsfs}
\usepackage{amscd}
\usepackage{tikz}
\usepackage{tikz-cd}
\usepackage{amsmath}
\usepackage{latexsym}
\usepackage{mathdots}
\usepackage{amsrefs}
\usepackage{graphicx}
\usepackage{mathtools}
\usepackage{leftidx}
\usepackage{tensor}
\usepackage{dsfont}
\usepackage{tikz-cd} 
\usepackage[new]{old-arrows}
\usepackage{xcolor}
\usepackage{bbm}
\usepackage{enumitem}
\usepackage{cleveref}

\allowdisplaybreaks

\theoremstyle{plain}
\newtheorem{theorem}{Theorem}[section]
\newtheorem{taggedtheoremx}{Theorem}
\newenvironment{taggedtheorem}[1]
 {\renewcommand\thetaggedtheoremx{#1}\taggedtheoremx}
 {\endtaggedtheoremx}
\newtheorem{corollary}[theorem]{Corollary}
\newtheorem{induction hypothesis}[theorem]{Induction Hypothesis}

\newtheorem{lemma}[theorem]{Lemma}
\newtheorem{proposition}[theorem]{Proposition}
\numberwithin{equation}{section}

\theoremstyle{definition}
 
   \newtheorem*{funding statement}{Funding statement}

\theoremstyle{remark}
\newtheorem{remark}[theorem]{Remark}

\newtheorem*{acknowledgements}{Acknowledgments}

\makeatletter
\newcommand{\mytag}[2]{%
  \text{#1}%
  \@bsphack
  \begingroup
    \@onelevel@sanitize\@currentlabelname
    \edef\@currentlabelname{%
      \expandafter\strip@period\@currentlabelname\relax.\relax\@@@%
    }%
    \protected@write\@auxout{}{%
      \string\newlabel{#2}{%
        {#1}%
        {\thepage}%
        {\@currentlabelname}%
        {\@currentHref}{}%
      }%
    }%
  \endgroup
  \@esphack
}

\begin{document}

\title[The local converse theorem]{The local converse theorem for odd special orthogonal and symplectic groups in positive characteristic}
\author[]{Yeongseong Jo}

\address{Department of Mathematics Education, Ewha Womans University, Seoul 03760, Republic of Korea}
\email{\href{mailto:yeongseongjo@outlook.com}{yeongseongjo@outlook.com};\href{mailto:yeongseong.jo@ewha.ac.kr}{yeongseong.jo@ewha.ac.kr}}


\begin{abstract}
Let $F$ be a non-archimedean local field of characteristic different from $2$ and $G$ be either an odd special orthogonal group ${\rm SO}_{2r+1}(F)$ or
a symplectic group ${\rm Sp}_{2r}(F)$. In this paper, we establish the local converse theorem for $G$. Namely, for given two irreducible admissible generic representations of $G$
with the same central character, if they have the same local gamma factors twisted by irreducible supercuspidal representations of ${\rm GL}_n(F)$ for all $1 \leq n \leq r$
with the same additive character, these representations are isomorphic. Using the theory of Cogdell, Shahidi, and Tsai on partial Bessel functions and the classification of irreducible generic representations, we break the barrier on the rank of twists $1 \leq n \leq 2r-1$ in the work of Jiang and Soudry, and extend the result of Q. Zhang, which was  achieved for all supercuspidal representations in characteristic $0$.
\end{abstract}

\maketitle

\section{Introduction}

Let $G$ be either an odd special orthogonal group ${\rm SO}_{2r+1}(F)$ or
a symplectic group ${\rm Sp}_{2r}(F)$ over a non-archimedean local field $F$ of characteristic different from $2$.
We study Rankin--Selberg integrals attached to a pair $(\pi,\tau)$ of irreducible admissible generic representations 
of $G$ and ${\rm GL}_n(F)$. The integrals are constructed by Soudry for $G={\rm SO}_{2r+1}(F)$ \cite{Sou93}, and
 by Gelbart and Piatetski-Shapiro with $r=n$ \cite{GP87}, and by Ginzburg, Rallis, and Soudry with $r \neq n$ \cite{GRS98} for $G={\rm Sp}_{2r}(F)$.
 We fix a nontrivial additive character $\psi$ of $F$. The local $\gamma$-factor $\gamma(s,\pi \times \tau,\psi)$ of $\pi \times \tau$ arises as a proportionality factor between Rankin--Selberg integrals or Shimura type integrals in the symplectic case.
 The purpose of this paper is to present the following type of the local converse theorem stated as local $\gamma$-factors in contrast to global $L$-functions appearing in the global converse theorem \cite[\S 3-\S 4]{Cog14}. Another new contribution in this paper is considering of all generic representations instead of generic supercuspidal representations.

\begin{taggedtheorem}{A}
\label{main}
Let $\pi$ and $\pi'$ be irreducible admissible generic representations of $G$ with the common central character $\omega$. If
 \[
 \gamma(s,\pi \times \tau,\psi)=\gamma(s,\pi' \times \tau,\psi)
 \]
 for all irreducible generic supercuspidal representations $\tau$ of ${\rm GL}_n(F)$ and for all $n$ with $1 \leq n \leq r$, then $\pi \cong \pi'$. 
\end{taggedtheorem}

In the ${\rm SO}_{2r+1}(F)$ case, the local converse theorem in characteristic zero has been proved in \cite{JS13} and \cite{Cha19} or \cite{JL18}.
In the ${\rm Sp}_{2r}(F)$ case, the local converse theorem for generic supercuspidal representations was proved in \cite{Zha18} by twisting generic representations of ${\rm GL}_n(F)$.
In this paper, we extend the above result to al generic representations of $G$ using fewer twists (only supercuspidal representations of ${\rm GL}_n(F)$) by
using the multiplicativity of gamma factors and by analyzing the poles and zeros of gamma factors following the strategies of \cite{JS13} and \cite{Zha18,Zha19}.
\par

Let $\widetilde{\rm Sp}_{2r}$ be the  metaplectic (double) cover of ${\rm Sp}_{2r}(F)$. In pioneering work of Jiang and Soudry \cite{JS03}, the local converse theorem has been settled for $G={\rm SO}_{2r+1}(F)$ twisted by ${\rm GL}_n(F)$ with $1 \leq n \leq 2r-1$
over $p$-adic field. To obtain this result, Jiang and Soudry combined global weak functoriality from ${\rm SO}_{2r+1}(F)$ to ${\rm GL}_{2r}(F)$ with local descent  \cite{JS13}  (so-called the {\it local backward lifting} \cite{JS03}) from ${\rm GL}_{2r}(F)$ to $\widetilde{\rm Sp}_{2r}$ and then the local theta correspondence from $\widetilde{\rm Sp}_{2r}$ to  ${\rm SO}_{2r+1}(F)$ to
be able to pull Henniart's local converse theorem \cite{Hen93} back from ${\rm GL}_{2r}(F)$ to ${\rm SO}_{2r+1}(F)$. As addressed in the introduction of Jacquet--Liu's paper \cite{JL18}, our main result is already well-known for ${\rm SO}_{2r+1}(F)$ over $p$-adic fields $F$, because it is reduced to the twisted up to ${\rm GL}_r(F)$ version local converse theorem for ${\rm GL}_{2r}(F)$ in the exact same way. 
This kind of local converse theorems for ${\rm GL}_{n}(F)$ 
has been first formulated by Jacquet, Piatetski-Shapiro, and Shalika \cite[\S 7]{Cog14} which is known as ``Jacquet's conjecture" \cite{Cha19,JL18,JS03,Zha18,Zha19} for decades and it is independently confirmed by Chai \cite{Cha19}, and Jacquet and Liu \cite{JL18}. Afterwords, Morimoto \cite{Mor18} adapted this approach to
local converse theorems for even unitary groups ${\rm U}_{2r}(F)$, which uses local/global functoriality and descent method.

\par
Another approach to the aforementioned converse theorem for ${\rm GL}_n(F)$ exploits Rankin--Selberg integrals and Fourier inversion formula,
whose method has its origin in the work of Henniart \cite{Hen93} and Chen \cite{Che06}. Subsequently, Jacquet and Liu \cite{JL18} generalized the method
to resolve Jacquet's conjecture.

\par 
We are guided by the analogy in modern number theory between number fields and (global) function fields. 
Based on our knowledge in number field cases, the surjectivity of  local functorial lifting maps by Lomel\'{\i} \cite[\S 9]{Lom09} should rely on the theory of local descent of Ginzburg, Rallis, and Soudry  \cite{JS13} in cases where $G$ is ${\rm SO}_{2r+1}(F)$ over local function fields $F$ (see \cite[Proof of Proposition 6.1]{JS03}).
To date, the local descent map from general linear groups to classical groups \cite{JS13} is not written up in this generality in the literature, while it is our belief that the local descent theory should continue to work over positive characteristics. 
This barrier is addressed by Gan and Lomel\'{\i} \cite{GL18} as well, when they outline that the surjectivity of conjectural local Langlands reciprocity maps from analytic sides to arithmetic (Galois) sides given in \cite[(7.2)]{GL18}
becomes one of the immediate consequences of that of local functorial transfer maps. Consequently, an alternative direction that they took at that time was to impose a working hypothesis \cite[\S 7.5 and Theorem 7.5]{GL18} weaker than the surjectivity, which is sufficient for the application therein. 
The local descent map turns into a stumbling block which still persists in the symplectic cases, as the irreducibility of the local descent map is amount to the local converse theorem \citelist{\cite{JS13}*{Theorem A7}\cite{Zha18}*{Introduction}}. Somewhat surprisingly, the equivalent statement is conditionally valid under hypothetical computations, and to the best of our knowledge the detailed proof might be included in their forthcoming paper \cite[p.770-771]{JS13} (cf. \cite[Introduction]{Zha18}).

\par
To be absolutely rigorous, we avoid any usage of local descent methods, and adapt the theory of partial Bessel functions developed by Cogdell, Shahidi, and Tsai \cite{CST17}.
 The partial Bessel function as a form of Howe Whittaker functions (also referred to {\it Howe vectors})
was first introduced by R. Howe \cite{Howe}. Utilized the theory of partial Bessel functions
and Howe vectors, Jacquet's conjecture is independently confirmed by Chai \cite{Cha19}.
 Later, the third method has been prominently refined and improved in a series of work of Q. Zhang \cite{Zha18,Zha19},
 when he settled down local converse theorem of ${\rm Sp}_{2r}(F)$ and ${\rm U}_{2r+1}(F)$ over $p$-adic fields \cite{Zha18,Zha19} as an extension of his Ph.D. thesis \cite{Zha17-1,Zha17-2}, but only for irreducible generic supercupidal representations. The article should belong to a spirit of paper \cite{Zha18,Zha19}
stemmed from the breakthrough work of Cogdell, Shahidi, and Tsai \cite{CST17}. Along the lines of this framework, it is noteworthy to point out that even special orthogonal cases are carried out by A. Hazeltine and Liu \cite{LH22} which is independent of the work of Haan, Kim, and Kwon \cite{HKK23}.

\par
 On top of that, there are clear differences between characteristic zero fields and positive characteristic fields, when dealing with local converse theorems.
 First, multiplicity one theorems needs to be established beforehand in order to define Rankin--Selberg $\gamma$-factors as proportionalities. 
 In the case of $n=r$, the uniqueness of Bessel and Fourier-Jacobi models for positive odd characteristic is recently proved by Mezer \cite{Mez}.
 Second, the definition of local $L$-functions for $G \times {\rm GL}_n(F)$ is unavailable for positive characteristic except even special orthogonal group cases  \cite{Kap13-2}, strictly speaking. 
 To get around this issue, we show that Rankin--Selberg $\gamma$-factors agree with corresponding Langlands--Shahidi $\gamma$-factors. In the time of proof,
 we do not know the existence of the involution of M{\oe}glin, Vign\'eras, and Waldspurger (cf. \cite[\S 5]{Kap13}) unlike characteristic zero fields. 
 For this reason, we directly compare $L$-factors with their counterpart dual $L$-factors over unramified places. 
 After that, we appeal to the holomorphy of Langlands--Shahidi $L$-factors. Third, the rich theory of Langlands--Shahidi method developed prominently by Lomel\'{\i} \cite{Lom09,Lom15,Lom16,Lom17} for positive characteristic fields is crucial in reducing local converse theorem to that for irreducible supercuspidal representations. In contrast to local base change lifts \cite{LH22,Mor18} in characteristic zero cases, which seem to be overkill, we take a straight path. Our strategy of proof is based on Casselman--Shahidi Lemma (Lemma \ref{CS-SO2n}), analytic properties of intertwining operators and Plancherel measure, and the classification of irreducible admissible generic representations of $G$ (Lemma \ref{SOcuspsupport}) accordingly. This leads us to extend the local converse theorem originally built for supercuspidal representations to all irreducible
admissible generic ones of $G$.

\par
As an application to local function fields of positive characteristic in particular, Theorem \ref{main} explains an injectivity of a local functorial lift or local transfer map of irreducible generic supercuspidal representations $\pi$ of $G$
to irreducible supercupidal representations $\Pi$ of ${\rm GL}_R(F)$ defined by Lomel\'{\i} \cite[\S 9]{Lom09} with $R=2r$ if $G={\rm SO}_{2r+1}(F)$ and $R=2r+1$ if $G={\rm Sp}_{2r}(F)$ (Corollaries \ref{SO2n+1-transfer} and \ref{SP2n-transfer}). 
In keeping the spirit of Langlands principle of functoriality described in \cite{JS04}, we expect that the supercuspidality condition can be dropped.
Finally, let us mention that the result of this paper will be used in a forthcoming work by Haan, Kim, and Kwon \cite{HKK23}
on the local converse theorem for even special orthogonal groups ${\rm SO}_{2r}(F)$ and orthogonal groups ${\rm O}_{2r}(F)$.

\par
The overview of this article is the following. Section \ref{RS} contains Rankin--Selberg integrals for ${\rm SO}_{2r+1}(F) \times {\rm GL}_n(F)$,
Bessel models, and unnormalized and normalized $\gamma$-factors. In addition, we show that normalized $\gamma$-factors coincide with the one defined via the Langlands--Shahidi method in Section \ref{SO2n+1-equal-section}.
The proof of local converse theorem for odd special orthogonal cases is devoted in Section \ref{SO2n+1-localconverse}, and symplectic cases are considered in Section \ref{LCSp2r}.

\section{The Local Converse Theorem for SO$(2r+1)$}
\label{RS}

Let $F$ be a non-archimedean local field of characteristic different from $2$.
The base field $F$ is a finite extension of $\mathbb{Q}_p$ or $\mathbb{F}_p((t))$, called a {\it $p$-adic field} in characteristic 0 or a {\it local function field}
in characteristic $p$. We let $\mathcal{O}$ be its ring of integers, and $\mathfrak{p}$ its maximal ideal. 
We let $\varpi$ denote a uniformizer, so that $\mathfrak{p}=(\varpi)$ and $\mathcal{O} \slash \mathfrak{p} \cong \mathbb{F}_q$ for some finite field of order $q$.
We normalize the absolute value by  
$|\varpi|=q^{-1}.$ Throughout this paper, despite stating the main result in the setting of non-archimedean local fields of arbitrary characteristic $\neq 2$,
we will give a detailed and complete proof only in the case of positive characteristic. The rationale for our choice is that 
the literature dealing with Rankin--Selberg methods over function fields is less developed.

\subsection{Structure theory of odd orthogonal groups}
For a positive integer $k$, we inductively define the element $J_k \in {\rm GL}_k(F)$ by
\[
  J_k=\begin{pmatrix} & 1 \\ J_{k-1} &  \end{pmatrix}; \quad J_1=(1).
\]
Let ${\rm SO}_k(F)$ be the special orthogonal group defined by
\[
 {\rm SO}_k(F)=\{ g \in {\rm GL}_k(F) \,|\, \det g=1, \prescript{t}{}{g}J_kg=J_k \}.
\]
Let $B_k=A_kU_k$ be the standard upper triangular Borel subgroup of ${\rm SO}_k(F)$ with the maximal torus $A_k$
and the standard maximal unipotent subgroup $U_k$. Let $Z_k=\{  I_{k}\}$ be the center of ${\rm SO}_k(F)$.

\par
For $k=2n$ even, we let $Q_k$ be the Siegel parabolic subgroup, which has the Levi decomposition
$Q_k=M_k \ltimes V_k$ with
\[
 M_k=\left\{ {\textbf m}_n(a):= \begin{pmatrix} a & \\ & a^{\ast} \end{pmatrix}  \,\middle|\, a \in {\rm GL}_n(F) \right\} \cong {\rm GL}_n(F)
\]
and 
\[
 V_k=\left\{ {\textbf u}_n(b):=\begin{pmatrix} I_n & b \\ & I_n  \end{pmatrix}  \in  {\rm SO}_k(F) \right\},
\]
where $a^{\ast}=J_n \prescript{t}{}{a}^{-1} J_n$. The condition ${\textbf u}_n(b) \in {\rm SO}_k(F)$ forces that $J_nb=-\prescript{t}{}{b}J_n$.

\par
For our convenience, we will abuse the notation by letting $B=B_{2r+1}$, $A=A_{2r+1}$, $Z=Z_{2r+1}$, and $U=U_{2r+1}$.
A prototype element $t \in A$ is of the form
\begin{equation}
\label{maximalelement}
 t= {\rm diag}(a_1,\dotsm,a_r,1,a^{-1}_r,\dotsm,a^{-1}_1), \quad a_1, \dotsm, a_r \in F^{\times}.
\end{equation}

\par
For $X \in {\rm GL}_r(F)$, we denote
\[
  X^{\wedge}=\begin{pmatrix} X && \\ &1& \\  && X^{\ast} \end{pmatrix} \in {\rm SO}_{2r+1}(F).
\]
For $n \leq r$, we denote 
\[
w_{n,r-n}=\begin{pmatrix} & I_n \\ I_{r-n} & \end{pmatrix}^{\wedge}.
\]
Then we have $w_{n,r-n}^{-1}=w_{r-n,n}$.
 We embed ${\rm SO}_{2n}(F) \rightarrow {\rm SO}_{2r+1}(F)$ via
\[
  j \begin{pmatrix} a & b \\ c & d \end{pmatrix}=\begin{pmatrix} I_{r-n} &&&& \\ &a&&b&\\  &&1&&\\ &c&&d&\\ &&&& I_{r-n} \\  \end{pmatrix};\quad  j^{w_{n-r,n}} \begin{pmatrix} a & b \\ c & d \end{pmatrix}
  =\begin{pmatrix} a &&b \\ &I_{2(r-n)+1}& \\ c&&d  \end{pmatrix}.
\]
The maps $j$ \cite[\S 7.1]{LZ22} and $j^{w_{n-r,n}}$ \citelist{\cite{JS03}*{p.772} \cite{Kap15}*{\S 3.1}} are related by $j^{w_{n-r,n}}(g)=j(w_{n-r,n}gw^{-1}_{n-r,n})$, and we regard an element of ${\rm SO}_{2n}(F)$ as an element of ${\rm SO}_{2r+1}(F)$ via the embedding $j$ without any further notice.

\subsection{Uniqueness of Bessel models}
Let $P=MN$ be the Parabolic subgroup of ${\rm SO}_{2r+1}(F)$ with the Levi subgroup
\[
 M=\{ {\rm diag}(a_1,\dotsm,a_{r-n},g,a^{-1}_{r-n},\dotsm,a^{-1}_1) \,|\, a_i \in F^{\times}, g \in {\rm SO}_{2n+1}(F) \} \cong (F^{\times})^{r-n} \times {\rm SO}_{2n+1}(F).
\]
On $N$, we introduce the character
\[
  \psi^{-1}_N(u):= \psi^{-1} \left( \sum_{i=1}^{r-n-1} u_{i,i+1}+u_{r-n,r+1} \right) \;\; \text{for} \;\; u=(u_{i,j})_{1 \leq i,j \leq 2r+1} \in N.
\]
We let
 \[
 H=
 \begin{cases}
 j({\rm SO}_{2n}(F)) \ltimes N & \text{for} \quad 1\leq n  < r, \\
  j({\rm SO}_{2r}(F)) & \text{otherwise.}
 \end{cases}
 \]
In the matrix realization, we have
 \[
   H=\left\{ \begin{pmatrix} u& \ast&\ast&\ast & \ast \\ & a & & b &\ast \\ &  & 1 &  & \ast \\  & c & & d &\ast \\  & & & & u^{\ast} \end{pmatrix} \, \middle|\, u \in U_{{\rm GL}_{r-n}(F)}, \begin{pmatrix} a & b \\ c & d \end{pmatrix} \in {\rm SO}_{2n}(F) \right\}.
 \]
There is a representation $\nu$ of $H$ such that $\nu|_N=\psi_U^{-1}|_N$, and $\nu|_{{\rm SO}_{2n}(F)}=1$ \cite[\S 15]{GGP12}.
The pair $(H,\nu)$ is called a {\it Bessel data} of ${\rm SO}_{2r+1}(F)$. For an irreducible representation $\sigma$ of ${\rm SO}_{2n}(F)$,
$\sigma$ can be regarded as a representation of $H$ via the natural quotient map $H \twoheadrightarrow  j({\rm SO}_{2n}(F))$.
 
 \begin{proposition}[Uniqueness of Bessel models] 
 \label{Bessel}
 Let $\pi$ be an irreducible smooth representation of ${\rm SO}_{2r+1}(F)$ and $\sigma$
an irreducible smooth representation of ${\rm SO}_{2n}(F)$.
 Then we have
 \[
  \dim {\rm Hom}_{H}(\pi,\sigma \otimes \nu) \leq 1.
 \]
 \end{proposition}

\begin{proof}
The equal rank case $(n=r)$ is treated very recently in \cite[Theorem 1.10]{Mez}. In virtue of \cite[Theorem 15.1]{GGP12}, the general case $(n < r)$ can be deduced from the equal rank
case $(n=r)$. In more general context, the one-dimensionality holds for finitely generated representations of Whittaker type \cite[Section 8]{Sou93} (cf. \cite[p.407]{Kap15}), owing to the theory of derivative of Bernstein and Zelevinsky \cite{BZ76}.
\end{proof}

The Frobenius reciprocity asserts that a non-zero element in ${\rm Hom}_{H}(\pi,\sigma \otimes \nu)$ induces an embedding $\pi \hookrightarrow {\rm Ind}^{{\rm SO}_{2r+1}(F)}_H(\sigma \otimes\nu)$, which is called a {\it Bessel model} of $\pi$ associated with the data $(H,\sigma \otimes \nu)$.

\subsection{The intertwining operator}
Let $n \leq r$ be a positive integer and let $U_{{\rm GL}_{n}(F)}$ denote the upper triangular unipotent subgroup of ${\rm GL}_{n}(F)$. We define the character $\psi_{U_{{\rm GL}_{n}(F)}}$ of $U_{{\rm GL}_{n}(F)}$ by
\[
  \psi^{-1}_{U_{{\rm GL}_{n}(F)}}(u):=\psi^{-1} \left( \sum_{i=1}^{n-1} u_{i,i+1} \right), \quad u=(u_{i,j})_{1 \leq i ,j \leq n} \in U_{{\rm GL}_{n}(F)}.
\]
We let $(\tau,V_{\tau})$ be an irreducible $\psi^{-1}_{U_{{\rm GL}_{n}(F)}}$-generic representation of  ${\rm GL}_n(F)$.
Given a complex number $s \in \mathbb{C}$, we let $\tau_s$ be the representation of $M_{2n}$ defined by
\begin{equation}
\label{tau-s}
  \tau_s({\textbf m}_n(a))=|\det a|^{s-1/2} \tau(a) \quad \text{for $a \in {\rm GL}_{n}(F)$.}
\end{equation}
We form the normalized induced representation $I(s,\tau)={\rm Ind}^{{\rm SO}_{2n}(F)}_{Q_{2n}}(\tau_{s} \otimes {\textbf 1}_{V_{2n}})$.
The space $I(s,\tau)$ may be then deemed as the space of smooth functions $\xi_s : {\rm SO}_{2n}(F) \rightarrow V_{\tau}$ that satisfy
\[
  \xi_s({\textbf m}_n(a){\textbf u}_n(b)g)=\delta^{1/2}_{Q_{2n}}({\textbf m}_n(a))\tau_s({\textbf m}_n(a))\xi_s(g)
\]
for any ${\textbf m}_n(a) \in M_{2n}$, ${\textbf u}_n(b) \in V_{2n}$, and $g \in {\rm SO}_{2n}(F)$, where $\delta_{Q_{2n}}$ is a modulus character.

\par
We denote by $\tau^{\ast}$ the representation of ${\rm GL}_n(F)$ defined by $\tau^{\ast}(a)=\tau(a^{\ast})$ for $a \in {\rm GL}_{n}(F)$.
If $\tau$ is irreducible, $\tau^{\ast}$ is isomorphic to the contragredient representation $\tau^{\vee}$ of $\tau$.
We fix a nonzero Whittaker functional $\lambda \in {\rm Hom}_{U_{{\rm GL}_{n}(F)}}(\tau,\psi^{-1}_{U_{{\rm GL}_{n}(F)}})$. 
The Whittaker functions associated to $v \in V_{\tau}$ are given by
\[
 W_v(a):=\lambda(\tau(a)v) \quad \text{and} \quad W^{\ast}_v(a):=\lambda(\tau(d_na^{\ast})v), 
\]
where $d_n={\rm diag}(-1,1,\dotsm,(-1)^n) \in {\rm GL}_n(F)$, and their Whittaker models are $\mathcal{W}(\tau,\psi^{-1}_{U_{{\rm GL}_{n}(F)}})$ and $\mathcal{W}(\tau^{\ast},\psi^{-1}_{U_{{\rm GL}_{n}(F)}})$, respectively.
Given $\xi_s \in I(s,\tau)$, we consider  the $\mathbb{C}$-valued function $f_{\xi_s}$ on ${\rm SO}_{2n}(F) \times {\rm GL}_n(F)$ defined by
\[
 f_{\xi_s}(g,a)=\lambda(\tau(a)\xi_s(g)) \quad \text{for $g \in {\rm SO}_{2n}(F)$ and $a \in {\rm GL}_n(F)$}.
\]
Let $V^{{\rm SO}_{2n}}_{Q_{2n}}(s,\mathcal{W}(\tau,\psi^{-1}_{U_{{\rm GL}_{n}(F)}}))$ be the space of functions $\{ f_{\xi_s} \,|\, \xi_s \in I(s,\tau) \}$. 
For $n \leq r$, we denote
\[
  w_0=\begin{pmatrix} & I_n \\ -I_n & \end{pmatrix} \in {\rm SO}_{2n}(F).
  \]
We define a (standard) intertwining operator 
\[
M(s,\tau,\psi^{-1}) : V^{{\rm SO}_{2n}}_{Q_{2n}}(s,\mathcal{W}(\tau,\psi^{-1}_{U_{{\rm GL}_{n}(F)}})) \rightarrow V^{{\rm SO}_{2n}}_{Q_{2n}}(1-s,\mathcal{W}(\tau^{\ast},\psi^{-1}_{U_{{\rm GL}_{n}(F)}}))
\]
 by
 \[
  M(s,\tau,\psi^{-1})f_s(g,a)=\int_{V_{2n}} f_s(w_0^{-1}vg,d_na^{\ast}) \,dv, \quad g \in {\rm SO}_{2n}(F), a \in {\rm GL}_n(F).
 \]
 This integral converges absolutely for ${\rm Re}(s)$ sufficiently large 
and is defined by a meromorphic continuation otherwise.

\subsection{The local zeta integrals and gamma factors}
We set $G={\rm SO}_{2r+1}(F)$
Let $\psi_{U}$ be the generic character of $U$ defined by
\[
  \psi_{U}(u)=\psi \left( \sum_{i=1}^r u_{i,i+1} \right), \quad u=(u_{i,j})_{1 \leq i,j \leq 2r+1} \in U.
\]
Let $(\pi,V_{\pi})$ be an irreducible $\psi_{U}$-generic representation of $G$, whose Whittaker model is $\mathcal{W}(\pi,\psi_{U})$.
Soudry \cite[\S 1]{Sou93} established the integral representation associated to $W \in \mathcal{W}(\pi,\psi_{U})$ and $f_s \in V^{{\rm SO}_{2n}}_{Q_{2n}}(s,\mathcal{W}(\tau,\psi^{-1}_{U_{{\rm GL}_{n}(F)}}))$
over non-archimedean local fields including positive characteristic of our cheif interest in the following way;
\[
  \Psi(W,f_s)=\int_{U_{2n} \backslash {\rm SO}_{2n}(F)} \int_{{\rm Mat}_{r-n \times n}(F)} W \left( w_{n,r-n} \begin{pmatrix} I_{r-n} & x \\  &I_n \end{pmatrix}^{\wedge} j(g)w^{-1}_{n,r-n} \right) f_s(g,I_n) \,dx dg.
\]
These integrals are absolutely convergent for $\mathrm{Re}(s)$ sufficiently large and enjoy meromorphic continuation. 

\begin{proposition}[Local functional equation]
There exists a meromorphic function $\Gamma(s,\pi \times \tau,\psi)$ in $\mathbb{C}(q^{-s})$ such that
\begin{equation}
\label{SO2n+1func}
\Psi(W,M(s,\tau,\psi^{-1})f_s)=\Gamma(s,\pi \times \tau,\psi) \Psi(W,f_s),
\end{equation}
for all $W \in \mathcal{W}(\pi,\psi_U)$ and $f_s \in V^{{\rm SO}_{2n}}_{Q_{2n}}(s,\mathcal{W}(\tau,\psi^{-1}_{U_{{\rm GL}_{n}(F)}}))$.
\end{proposition}

\begin{proof}
We denote by $\pi^{w_{n,r-n}}$ the representation $g \mapsto \pi(w_{n,r-n}gw^{-1}_{n,r-n})$ of $G$. For all but a discrete subset of $s$, the meromorphic continuation of the integrals
$\Psi(W,f_s)$ and $\Psi(W,M(s,\tau,\psi^{-1})f_s)$ can be regarded as bilinear forms $(W,f_s) \mapsto \Psi(W,f_s)$ and $(W,f_s) \mapsto \Psi(W,M(s,\tau,\psi^{-1})f_s)$ in the space
\begin{equation}
\label{BesselFunction}
 {\rm Hom}_{H}(\pi^{w_{n,r-n}}\otimes V^{{\rm SO}_{2n}}_{Q_{2n}}(s,\mathcal{W}(\tau,\psi^{-1}_{U_{{\rm GL}_{n}(F)}})), \nu).
\end{equation}
Since the induced representation $V^{{\rm SO}_{2n}}_{Q_{2n}}(s,\mathcal{W}(\tau,\psi^{-1}_{U_{{\rm GL}_{n}(F)}}))$ is irreducible outside a finite number of $q^s$, the uniqueness of Bessel function,
 Proposition \ref{Bessel}, implies that the dimension of the space \eqref{BesselFunction} is at most one outside a finite number of $q^s$.
 Then the local $\gamma$-factor $\gamma(s,\pi \times \tau,\psi)$ is defined as a proportionality. 
\end{proof}

The normalized intertwining operator is
\[
  N(s,\tau,\psi^{-1})=C^{{\rm SO}_{2n}}_{\psi}(s-1/2,\tau,w_0)M(s,\tau,\psi^{-1}),
\]
where $C^{{\rm SO}_{2n}}_{\psi}(s-1/2,\tau,w_0)$ is the {\it Langlands--Shahidi $\gamma$-factors} defined by Shahidi's functional equation \citelist{\cite{Kap15}*{(3.2)}\cite{Lom15}*{\S 2.2}  \cite{Lom16}*{\S 6.4 Theorem}};
\begin{multline}
\label{exterior-localcoeff-def}
 \int_{V_{2n}} f_s({\textbf m}_n(d_n)w_0^{-1}v,I_n) \psi(v_{n-1,n+1}) \,dv\\
 = C^{{\rm SO}_{2n}}_{\psi}(s-1/2,\tau,w_0) \int_{V_{2n}} M(s,\tau,\psi^{-1})f_s ({\textbf m}_n(d_n)w_0^{-1}v,I_n) \psi(v_{n-1,n+1}) \,dv.
\end{multline}
The local Rankin--Selberg gamma factor $\gamma(s,\pi \times \tau,\psi)$ is defined by the identity
\[
 \Psi(W,N(s,\tau,\psi^{-1})f_s)=\omega_{\tau}(-1)^r\gamma(s,\pi \times \tau,\psi) \Psi(W,f_s).
\]
As a result, we have the formula
\begin{equation}
\label{SO2n+1Unnormal}
\Gamma(s,\pi \times \tau,\psi)=\omega_{\tau}(-1)^r \frac{\gamma(s,\pi \times \tau,\psi)}{C^{{\rm SO}_{2n}}_{\psi}(s-1/2,\tau,w_0)},
\end{equation}
which implies that $\gamma(s,\pi \times \tau,\psi)$ is a rational function in $q^{-s}$. The motivation underlying the definition of $\gamma(s,\pi \times \tau,\psi)$ is 
to obtain the ``concise" multiplicative properties illustrated in Theorem \ref{SO2n+1property}-\ref{SO2n+1propert-4}.

\subsection{The equality of gamma factors} 
\label{SO2n+1-equal-section}
From this point on, $k$ will denote a (global) function a field with field of constant $\mathbb{F}_q$ and ring of ad\`eles $\mathbb{A}_k$.
If $r=0$, we set $G=\{ 1\}$. What we would like to do next is to explain that $\gamma(s,\pi \times \tau,\psi)$
satisfies a list of fundamental properties, which reconcile $\gamma$-factors from two exotic methods. The formulation below is adapted from Lomel\'{\i} \cite[\S 1.4]{Lom15}
who states it for the $\gamma$-factor arising from Langlands--Shahidi method.

\begin{theorem} Let $\pi$ and $\tau$ be a pair of irreducible admissible generic representations of ${\rm SO}_{2r+1}(F)$ and ${\rm GL}_n(F)$.
The $\gamma$-factor $\gamma(s,\pi \times \tau,\psi)$ satisfies the following properties.
\label{SO2n+1property}
 \begin{enumerate}[label=$(\mathrm{\roman*})$]
\item {\rm (Naturality)}\label{SO2n+1propert-1} Let $\eta : F' \rightarrow F$ be an isomorphism of local fields. Via $\eta$, $\pi$ and $\tau$ define 
irreducible admissible generic representations $\pi'$ of ${\rm SO}_{2r+1}(F')$ and $\tau'$ of ${\rm GL}_n(F')$, and a nontrivial additive character $\psi'$ of $F'$.
Then we have
\[
 \gamma(s,\pi \times \tau,\psi)=\gamma(s,\pi' \times \tau',\psi').
\]
\item {\rm (Isomorphism)}\label{SO2n+1propert-2} If $\pi'$ and $\tau'$ are irreducible admissible generic representations of ${\rm SO}_{2r+1}(F)$ and of ${\rm GL}_n(F)$ such that $\pi \cong \pi'$
and $\tau \cong \tau'$, then 
\[
 \gamma(s,\pi \times \tau,\psi)=\gamma(s,\pi' \times \tau',\psi).
\]
\item {\rm (Minimal cases)}\label{SO2n+1propert-3} For $r=0$, $\gamma(s,\pi \times \tau,\psi)=1$.
\item\label{SO2n+1propert-4} {\rm (Multiplicativity)} Let $\pi$ $($respectively, $\tau$$)$ be the irreducible generic quotient of a representation parabolically induced from
\[
 {\rm Ind}^{{\rm SO}_{2r+1}(F)}_{\mathcal{P}_{2r+1}}(\pi_1 \otimes \pi_2 \otimes \dotsm \otimes \pi_t \otimes \pi_0)
\]
$($respectively $ {\rm Ind}^{{\rm GL}_{n}(F)}_{\mathcal{P}_n}(\tau_1 \otimes \tau_2 \otimes \dotsm \otimes \tau_s)$$)$, where  $\mathcal{P}_{2r+1}$ $($respectively $\mathcal{P}_n$$)$  is the standard  
parabolic subgroup with its Levi part isomorphic to $\prod_{i=1}^t {\rm GL}_{r_i}(F) \times {\rm SO}_{2r_0+1}(F)$ 
$($respectively $\prod_{j=1}^s {\rm GL}_{n_j}(F)$$)$ associated to the partition $(r_1,\dotsm,r_t,r_0)$ of r $($respectively $(n_1,n_2,\dotsm,n_s)$ of $n$$)$. Then
\[
\gamma(s,\pi \times \tau,\psi)=\prod_{j=1}^s \gamma(s,\pi_0 \times \tau_j,\psi) \prod_{i=1}^t \prod_{j=1}^s \gamma(s,\pi_i \times \tau_j,\psi)  \gamma(s,\pi^{\vee}_i \times \tau_j,\psi).
\]
Here $\gamma(s,\pi_i \times \tau_j,\psi)$ is the ${\rm GL}_{r_i} \times {\rm GL}_{n_j}$ $\gamma$-factor defined by Jacquet, Piatetski-Shapiro and Shalika (cf. \cite{Cog14}).
\item {\rm (Dependence of $\psi$)}  \label{SO2n+1propert-5} For any $a \in F^{\times}$, let $\psi^a$ be the character given by $\psi^a(x)=\psi(ax)$. Then
\[
 \gamma(s,\pi \times \tau,\psi^a)=\omega_{\tau}(a)^{2r}|a|^{2nr(s-\frac{1}{2})}\gamma(s,\pi \times \tau,\psi).
\]
\item {\rm (Stability)} \label{SO2n+1propert-6} Let $\pi_1$ and $\pi_2$ be irreducible admissible generic representations of ${\rm SO}_{2r+1}(F)$.
For $\chi$ sufficiently highly ramified, we have
\[
 \gamma(s,\pi_1 \times \chi,\psi)= \gamma(s,\pi_2 \times \chi,\psi).
\]
\item {\rm (Global functional equation)} \label{SO2n+1propert-7}  
We assume that $\pi$ and $\tau$ are globally generic irreducible cuspidal automorphic representations of ${\rm SO}_{2r+1}(\mathbb{A}_k)$ and ${\rm GL}_n(\mathbb{A}_k)$.
Let $S$ be a finite set of places such that for all $v \notin S$, all data are unramified. Then
\[
  L^S(s,\pi \times \tau)=\prod_{v \in S} \gamma(s,\pi_v \times \tau_v,\psi_v) L^S(1-s,\pi^{\vee} \times \tau^{\vee}),
\]
where the partial $L$-function is given by
\begin{equation}
\label{partial-L-function}
  L^S(s,\pi \times \tau)=\prod_{v \notin S} L(s,\pi_v \times \tau_v).
\end{equation}
\end{enumerate}
\end{theorem}


\begin{proof} The properties \ref{SO2n+1propert-1}, \ref{SO2n+1propert-2} and \ref{SO2n+1propert-3} are straightforward.
As in \cite{Kap13,Kap15}, the hard part should be multiplicativity \ref{SO2n+1propert-4}, which follows from Soudry \cite{Sou93,Sou20}. The dependency on $\psi$ \ref{SO2n+1propert-5} is explained in \cite[\S 5.1]{Kap13}.
Stability \ref{SO2n+1propert-6} is confirmed by Cogdell and Piatetski-Shapiro \cite{CPS98}.
With regard to functional equation \ref{SO2n+1propert-7}, the involution of M{\oe}glin, Vign\'eras, and Waldspurger used in \cite[\S 5]{Kap13} is not recorded for this generality in the literature. 
For the sake of brevity, we include an alternative approach. Let $\pi_v={\rm Ind}^G_{B}(\mu)$ and $\tau_v={\rm Ind}^{{\rm GL}_n(F)}_{B_{n}}(\chi)$ be irreducible generic unramified representations of $G$ and ${\rm GL}_n(F)$,
where $\mu$ and $\chi$ are unramified characters of $B$ and $B_{n}$. 
If
\[
 t={\rm diag}(a_1,\dotsm,a_r,1,a^{-1}_r,\dotsm,a^{-1}_1) \in A
 \] 
 as in \eqref{maximalelement}, then $\mu(t)=\mu_1(t_1)\mu_2(t_2) \dotsm\mu_r(t_r)$. Let
 \[
   A_{\pi_v}={\rm diag}(\mu_1(\varpi),\dotsm,\mu_r(\varpi),\mu^{-1}_r(\varpi),\dotsm,\mu^{-1}_1(\varpi)) \in {\rm Sp}_{2r}(\mathbb{C})
 \]
 be the semisimple conjugacy class of ${\rm Sp}_{2r}(\mathbb{C})$ associated with $\pi_v$. 
 Since the Satake parameters $A_{\pi_v}$ and $A_{\pi_v^{\vee}}$ of $\pi_v$ and $\pi_v^{\vee}$, respectively, are conjugate in the dual group ${\rm Sp}_{2r}(\mathbb{C})$,
 it follows that $L(s,\pi_v \times \tau_v)=L(s,\pi_v^{\vee} \times \tau_v)$. To be specific, 
 we attach to $\tau$ a semisimple conjugacy class in ${\rm GL}_n(\mathbb{C})$
 denoted by
 \[
  B_{\tau_v}={\rm diag}(\chi_1(\varpi),\dotsm,\chi_n(\varpi)).
 \]
 Then $L$-functions for unramified principal series representations are defined by
\begin{multline*}
   L(s,\pi_v \times \tau_v)=\\
   \frac{1}{\det(I_{2rn}-(A_{\pi_v} \otimes B_{\tau_v})q^{-s})}
   =\prod_{1 \leq i \leq r,1\leq j \leq n} \left( \frac{1}{1-\mu_i(\varpi)\chi_j(\varpi)q^{-s}} \cdot
   \frac{1}{1-\mu^{-1}_i(\varpi)\chi_j(\varpi)q^{-s}} \right).
\end{multline*}
 However, the contragredient representation $\pi_v^{\vee}$ is ${\rm Ind}^{G}_{B}(\mu^{-1})$ (cf. \cite[Chapter 4-\S 4 Fact (1)]{Kim04}). The Satake parametrization
 gives semisimple conjugacy classes $\{ A_{\pi_v^{\vee}}=A_{\pi_v}^{-1} \}$ in ${\rm Sp}_{2r}(\mathbb{C})$. All that remains is to state that
 \[
   L(s,\pi_v^{\vee} \times \tau_v)=\frac{1}{\det(I_{2rn}-(A^{-1}_{\pi_v} \otimes B_{\tau_v})q^{-s})} =\frac{1}{\det(I_{2rn}-(A_{\pi_v} \otimes B_{\tau_v})q^{-s})}=L(s,\pi_v \times \tau_v). 
 \] 
 Based on unramified calculation of Rankin--Selberg integrals by Soundry \cite[\S 12]{Sou93}, Kaplan \cite[Theorem 1-(9)]{Kap15} obtains the following crude functional equation from the properties of Eisenstein series;
 \[
  L^S(s,\pi \times \tau)=\prod_{v \in S} \gamma(s,\pi_v \times \tau_v,\psi_v) L^S(1-s,\pi \times \tau^{\vee}).
\]
We apply the equality $L(s,\pi_v \times \tau_v)=L(s,\pi_v^{\vee} \times \tau_v)$ to $\tau_v^{\vee}$ for each $v \notin S$ to get 
\[
L^S(1-s,\pi \times \tau^{\vee})= L^S(1-s,\pi^{\vee} \times \tau^{\vee}). \qedhere
\]
\end{proof}

In contrast to \cite[\S 1.4-(v)]{Lom15}, the central character $\omega_{\pi}$ of $\pi$ in Theorem \ref{SO2n+1property}-\ref{SO2n+1propert-5} is absent, because the center $Z$ is $\{  I_{2r+1}\}$.
Regarding Theorem \ref{SO2n+1property}-\ref{SO2n+1propert-7}, the keen reader may notice that $\pi^{\vee}$ is used in \cite[\S 1.4-(vii)]{Lom15} as opposed to $\pi$ in \cite[Theorem 1-(9)]{Kap13}.
As mentioned in \cite[Theorem 1.5]{Lom15}, any system of properties \ref{SO2n+1propert-1}-\ref{SO2n+1propert-7} uniquely determines $\gamma$-factors.

\begin{theorem}
\label{SO2n-equality} 
The Rankin--Selberg $\gamma$-factor $\gamma(s,\pi \times \tau,\psi)$ coincides with Shahidi's $\gamma$-factor for $\pi \times \tau$ defined in \cite{Lom15}.
\end{theorem}

Theorem \ref{SO2n-equality} follows from the standard globalization \cite[Proposition 3.1]{Lom15} by Henniart and Lomel\'{\i}, which is a subtle refinement of a result of Henniart and Vign\'eras. The main point is that an irreducible supercuspidal generic representation can be embedded as a local factor of an irreducible cuspidal automorphic globally generic representation.

\subsection{The local converse theorem}
\label{SO2n+1-localconverse}

We begin with the case of irreducible generic supercuspidal representations. 
Analogues to main results in \cite{Zha18,Zha19} utilize twisting by all generic representations of ${\rm GL}_k(F)$ with $1 \leq k \leq r$.
We reduce number of twists by only using generic supercuspidal representations of ${\rm GL}_k(F)$.

\begin{theorem}
\label{SO-converse-supercusp}
Let $\pi$ and $\pi'$ be irreducible generic supercuspidal representations of $G$. If
 \[
 \gamma(s,\pi \times \tau,\psi)=\gamma(s,\pi' \times \tau,\psi)
 \]
 for all irreducible generic supercuspidal representations $\tau$ of ${\rm GL}_k(F)$ and for all $k$ with $1 \leq k \leq r$, then $\pi \cong \pi'$. 
\end{theorem}

\begin{proof}
We can ease the given assumption to
all irreducible generic smooth representations $\tau$ of ${\rm GL}_k(F)$ and for all $k$ with $1 \leq k \leq r$ by the multiplicativity of local
gamma factors, Proposition \ref{SO2n+1property}-\ref{SO2n+1propert-4}. As seen in \eqref{SO2n+1Unnormal}, a normalized gamma factor $\gamma(s,\pi \times \tau,\psi)$
and an unnormalized one $\Gamma(s,\pi \times \tau,\psi)$ differ by factors which only depend on $\tau$.
Therefore, $\gamma(s,\pi \times \tau,\psi)$ and $\gamma(s,\pi' \times \tau,\psi)$ in the assumption can be safely replaced with 
$\Gamma(s,\pi \times \tau,\psi)$ and $\Gamma(s,\pi' \times \tau,\psi)$. Combining all together, we conclude that
$\Gamma(s,\pi \times \tau,\psi)$ and $\Gamma(s,\pi' \times \tau,\psi)$ are equal for all irreducible generic smooth representations $\tau$ of ${\rm GL}_k(F)$ 
 and for all $k$ with $1 \leq k \leq r$. In the same argument as in the proofs of \cite{Zha18} and \cite{Zha19} using the theory of partial Bessel functions and Howe vectors,
 we can prove that $\pi \cong \pi'$. We omit the rest of the proof.
\end{proof}

To proceed, we need what is widely accepted as {\it Casselman--Shahidi Lemma} \cite[Theorem 5.1]{CS98}.
It is quite a deep theorem and is recently resolved by Luo \cite{Luo24} for the characteristic zero case.
In the case of local function fields, we do not strive for maximal generality, similar to Luo \cite{Luo24}, though the statement is believed to be true regardless of representations being supercuspidal.
Occasionally, the hypothesis might not be necessary, but which hold in all our applications we want at least.
Aided by \citelist{\cite{Lom15}*{\S 2.2}  \cite{Lom16}*{\S 6.4 Theorem}}, the Langlands--Shahidi local coefficient can be computed in different ways according to the following cases;
\[
C^{{\rm SO}_{2n}}_{\psi}(s-1/2,\tau,w_0)=
 \begin{cases}
 \gamma(s-1/2,\tau,\psi), \quad  &\text{if $\tau$ is a character of $F^{\times}$;} \\
 \gamma(2s-1,\tau,\wedge^2,\psi), \quad & \text{otherwise,} \\
 \end{cases}
\]
where $\gamma(s,\tau,\psi)$ denote a Tate $\gamma$-factor and  $\gamma(s,\tau,\wedge^2)$ is a Langlands--Shahidi exterior square $\gamma$-factor in
\citelist{ \cite{Lom15}*{\S 2.2} \cite{Lom16}*{\S 6.4 Theorem}}. For this reason, we deal with $n=1$ case separately. The Langlands--Shahidi exterior square $\gamma$-factor is in accordance with the local Langlands correspondence in the case of positive characteristic \cite[Theorem 2.2]{Lom15} and of characteristic zero \cite{CST17}.
In light of \cite{Jo23}, the local exterior square $L$-function $L(s,\tau,\wedge^2)$ appearing in this paper can be equivalently defined by the Langlands--Shahidi method
or the local zeta integrals of Jacquet--Shalika.

\begin{lemma}[Casselman--Shahidi Lemma]
\label{CS-SO2n}
Let $\tau$ be either an irreducible unitary supercuspidal representation of ${\rm GL}_n(F)$ with $n > 1$ or a unitary ramified character of $F^{\times}$. The intertwining operator
\[
  \frac{1}{L(2s-1,\tau,\wedge^2)} M(s,\tau,\psi^{-1})
\]
is entire as a function of $s$.
\end{lemma}

\begin{proof}
The statement of the lemma is clearly invariant under twisting by an unramified character.
Upon twisting $\tau$ by $|\det \cdot |^{it}$, it is enough to consider the pole on the real line.
We handle the case when $\tau$ is an irreducible unitary supercuspidal representation of ${\rm GL}_n(F)$ with $n > 1$.
With \eqref{tau-s} in mind, we caution that $s$ in Lomel\'{\i} \cite{Lom17} and Shahidi \cite{Sha90}
is really $s+1/2$ in our notation and Kaplan's \cite{Kap15}. 
To be consistent, we set $s'=s-1/2$.
The proof of \cite[Theorem 5.2]{CS98} applies verbatim to the case of ${\rm Re}(s') > 0$ or ${\rm Re}(s') < 0$.
As a result, we draw the conclusion (cf. the statement in the course of proof of \cite[Lemma 4.3]{LM17}) that
the possible real pole of $M(s',\tau,\psi^{-1})$ is at $s'=0$ and it can only occur if $\tau$ is self-dual, i.e., $\tau \cong \tau^{\vee}$.
Suppose therefore that $\tau$ is self-dual.
By using \eqref{exterior-localcoeff-def} twice, the definition of Plancherel measure leads us to
\begin{equation}
\label{Plancherel-composition}
\gamma(-2s',\tau^{\vee},\wedge^2,\psi^{-1}) M(-s',\tau^{\vee},\psi^{-1}) \gamma(2s',\tau,\wedge^2,\psi)M(s',\tau,\psi^{-1})={\rm Id}. 
\end{equation}
Owing to \cite[Proposition 5.4]{Lom17}, the induced representation $I(s',\tau)$ is irreducible if and only if $L(s',\tau,\wedge^2)$ 
has a pole at $s'=0$. If this is the case, the image of  $I(s',\tau)$ under the left hand side of maps in \eqref{Plancherel-composition} is either trivial or itself.
On the one hand, $L(1-2s',\tau,\wedge^2)$ has neither zeros nor poles for ${\rm Re}(s') < 1/2$, while $L(1+2s',\tau^{\vee},\wedge^2)$ has
no zeros and no poles for ${\rm Re}(s') > -1/2$. On the other hand, ${\rm Id}$ is holomorphic and always non-zero for $s \in \mathbb{C}$.
Combining those facts together, we obtain the holomorphy of normalized intertwining operators that we seek for.
\par
When $\tau$ is a unitary ramified character of $F^{\times}$, the proof is similar to that of Lemma \ref{CS-SP2n}. Since the metapletic case is more involved, we omit the proof entirely.
\end{proof}

Several versions of Jacquet's subquotient theorem \cite{Jac75} are established in various places such as \cite[Proposition 3.19, p.34]{BZ76} for the ${\rm GL}_n$ case and \cite[\S 2.1 Claim, p.55]{Ber92} for the connective reductive case.   Our case is the positive characteristic analogue of \cite[p.777]{JS03} and we use this occasion to summarize it in the following way.

\begin{lemma}
\label{SOcuspsupport}
Let $\pi$ be an irreducible admissible generic representation of ${\rm SO}_{2r+1}(F)$. Then $\pi$ is a subquotient of 
the unitary induced representation 
\[
  {\rm Ind}^{{\rm SO}_{2r+1}(F)}_{\mathcal{P}_{2r+1}}(\rho_1|\det|^{z_1} \otimes \dotsm \otimes \rho_t|\det|^{z_t}\otimes \pi_0),
\]
where $\mathcal{P}_{2r+1}$ is the standard parabolic subgroup with its Levi part isomorphic to $\prod_{i=1}^t {\rm GL}_{r_i}(F) \times {\rm SO}_{2r_0+1}(F)$
associated to the partition $(r_1,\dotsm,r_t,r_0)$ of $r$, each $\rho_i$ is an irreducible unitary supercuspidal representation of ${\rm GL}_{r_i}(F)$
with the $z_i$ real and ordered so that $z_1 \geq z_2 \geq \dotsm \geq z_t \geq 0$, and $\pi_0$ is an irreducible generic supercuspidal representation of $ {\rm SO}_{2r_0+1}(F)$.
\end{lemma}

We oftentimes call a triplet $(\mathcal{P}_{2r+1};\rho_1,\rho_2,\dotsm,\rho_t;\pi_0)$ {\it supercuspidal support} and $(z_1,z_2,\dotsm,z_t)$ {\it exponents} associated to the irreducible admissible generic representation $\pi$. We record the following simple but useful results from \cite[Proposition 2.9]{DNS15} and Lemma \ref{CS-SO2n}.

\begin{proposition} 
\label{PoleDecomposition}
Let $\tau$ be an irreducible unitary supercuspidal representation of ${\rm GL}_n(F)$. With the notation in Lemma \ref{SOcuspsupport}, we have the followings;
\begin{enumerate}[label=$(\mathrm{\roman*})$]
\item\label{PoleDecomposition-1} If the product $\prod_{i=1}^t \gamma(s+z_i,\rho_i \times \tau,\psi)$ has a real pole $($respectively, a real zero$)$ at $s=s_0$, then $\tau \cong \rho_i^{\vee}$
and $s_0=1-z_i$ $($respectively, $s_0=-z_i)$ for some $1 \leq i \leq t$.
\item\label{PoleDecomposition-2} If the product $\prod_{i=1}^t \gamma(s-z_i,\rho_i^{\vee} \times \tau,\psi)$ has a real pole $($respectively, a real zero$)$ at $s=s_0$, then $\tau \cong \rho_i$
and $s_0=1+z_i$ $($respectively, $s_0=z_i)$ for some $1 \leq i \leq t$.
\item\label{PoleDecomposition-3} The factor $\gamma(s,\pi_0 \times \tau,\psi)$ has no zero for ${\rm Re}(s) > 0$. 
 If $\gamma(s,\pi_0 \times \tau,\psi)$ has a real pole at $s=s_0$, then the pole must be a simple pole at $s_0=1$ and $\tau \cong \tau^{\vee}$. 
\end{enumerate}
\end{proposition}

\begin{proof}
The properties \ref{PoleDecomposition-1} and \ref{PoleDecomposition-2} are explained in \cite[Proposition 2.9]{DNS15}.  On account of Theorem \ref{SO2n-equality}, the first assertion of \ref{PoleDecomposition-3} follows from the holomorphy of tempered $L$-functions via Langlands-Shahidi methods \cite[\S 4.1-(ix)]{Lom15}.
With the definition of $\gamma(s,\pi_0 \times \tau,\psi)$ in hand, we have
\begin{multline*}
 \gamma(s,\pi_0 \times \tau,\psi) \Psi(W,f_s)
 =\omega_{\tau}(-1)^r\gamma(2s-1,\tau,\wedge^2,\psi)\Psi(W,M(s,\tau,\psi^{-1})f_s)\\
 =\omega_{\tau}(-1)^r\varepsilon(2s-1,\tau,\wedge^2,\psi) L(2(1-s),\tau,\wedge^2) \frac{\Psi(W,M(s,\tau,\psi^{-1})f_s)}{L(2s-1,\tau,\wedge^2)}.
\end{multline*}
Now \cite[Section 6]{Sou93} ensures that there exist a Whittaker function $W \in \mathcal{W}(\pi_0,\psi_{U})$ and a function $f_s\in  V^{{\rm SO}_{2n}}_{Q_{2n}}(s,\mathcal{W}(\tau,\psi^{-1}))$
such that $\Psi(W,f_s)$ is a non-zero constant. In virtue of Lemma \ref{CS-SO2n}, the existence of poles of $\gamma(s,\pi_0 \times \tau,\psi)$ entails to the pole of $L(2(1-s),\tau,\wedge^2)$ at $s=1$.
Rephrasing it in a different manner, $L(s,\tau,\wedge^2)$ has a pole at $s=0$. Taking into account the equality $L(s,\tau,\wedge^2)L(s,\tau,{\rm Sym}^2)=L(s,\tau \times \tau)$ \cite[Lemma 7.12]{Lom09}, $\tau$ becomes self-dual, which confirms the second part of \ref{PoleDecomposition-3}.
\end{proof}

\begin{remark}
In the course of the proof of Proposition \ref{PoleDecomposition}, we check that if  $\gamma(s,\pi_0 \times \tau)$ has a pole at $s=1$, then $L(s,\tau,\wedge^2)$ has a pole at $s=0$.
\end{remark}

By determining the location of poles and zeros, we delineate the internal structure of irreducible admissible generic representations in terms of supercupsidal supports.
In the identical argument as in \cite[Theorem 5.1]{JS03} and \cite[Theorem 3.5]{Liu11}, we can prove the following proposition. Henceforth we skip the proof entirely.

\begin{proposition}
\label{SO-factorization}
 With the notation in Lemma \ref{SOcuspsupport}, suppose that $\rho'_i$ are irreducible unitary supercuspidal representations of ${\rm GL}_{r_i}(F)$ for $1 \leq i \leq t'$
 and $\pi'_0$ is an irreducible generic supercuspidal representation of ${\rm SO}_{2r_0+1}(F)$ with $r=r'_0+\sum_{i=1}^{t'}r'_i$ and that $z'_1 \geq z'_2 \geq \dotsm \geq z'_{t'} \geq 0$
 are real numbers. Suppose also that
\begin{multline}
\label{GeneralDec}
\left[\prod_{i=1}^t \gamma(s+z_i,\rho_i \times \tau,\psi)\gamma(s-z_i,\rho_i^{\vee} \times \tau,\psi) \right]\gamma(s,\pi_0 \times \tau,\psi)\\
 =\left[\prod_{i=1}^{t'} \gamma(s+z'_i,\rho'_i \times \tau,\psi)\gamma(s-z'_i,{\rho'_i}^{\vee} \times \tau,\psi) \right]\gamma(s,\pi'_0 \times \tau,\psi)
\end{multline}
for all irreducible unitary supercuspidal representations $\tau$ of ${\rm GL}_n(F)$ with $1 \leq n \leq r$. Then $t=t'$ and there exists a permutation $\sigma$
of $\{ 1,2,\dotsm,t\}$ such that
\begin{enumerate}[label=$(\mathrm{\roman*})$]
\item $r_i=r'_{\sigma(i)}$ for all $i=1,2,\dotsm,t$;
\item $z_i=z'_{\sigma(i)}$ and $\rho_i \cong \rho_{\sigma(i)}$ for all $i=1,2,\dotsm,t$;
\item $\gamma(s,\pi_0 \times \tau,\psi)=\gamma(s,\pi'_0 \times \tau,\psi)$ for all irreducible unitary supercuspidal representations $\tau$ of ${\rm GL}_n(F)$ with $1 \leq n \leq r$. 
\end{enumerate}
\end{proposition}

We now come to the main result of the local converse theorem over local function fields, whose technique can be uniformly adapted to the corresponding result over $p$-adic fields.
This seems to recover and strengthen the work of Jiang and Soudry \cite{JS03}. 

\begin{theorem}
Let $\pi$ and $\pi'$ be irreducible admissible generic representations of $G$. If
 \[
 \gamma(s,\pi \times \tau,\psi)=\gamma(s,\pi' \times \tau,\psi)
 \]
 for all irreducible generic supercuspidal representations $\tau$ of ${\rm GL}_n(F)$ and for all $n$ with $1 \leq n \leq r$, then $\pi \cong \pi'$. 
\end{theorem}

\begin{proof}
We assume that $\pi$ is a subquotient of the normalized induced representation 
\[
 \Pi= {\rm Ind}^{{\rm SO}_{2r+1}(F)}_{\mathcal{P}_{2r+1}}(\rho_1|\det|^{z_1} \otimes \dotsm \otimes \rho_t|\det|^{z_t}\otimes \pi_0),
\]
with $z_1 \geq z_2 \geq \dotsm \geq z_t \geq 0$ and  $\pi'$ is a subquotient of the normalized induced representation 
\[
  \Pi'=  {\rm Ind}^{{\rm SO}_{2r+1}(F)}_{\mathcal{P}'_{2r+1}}(\rho'_1|\det|^{z'_1} \otimes \dotsm \otimes \rho'_{t'}|\det|^{z'_{t'}}\otimes \pi'_0),
\]
with $z'_1 \geq z'_2 \geq \dotsm \geq z'_{t'} \geq 0$. In light of Proposition \ref{SO2n+1property}-\ref{SO2n+1propert-4}, it follows from Theorem \ref{SO-converse-supercusp} coupled with Proposition \ref{SO-factorization} that $\Pi$ is equivalent to $\Pi'$. We derive our conclusion from the uniqueness of the generic constituent in the induced representation $\Pi \cong \Pi'$.
\end{proof}

With Theorem \ref{SO2n-equality} in mind, we obtain the injectivity of the local Langlands lift or transfer.

\begin{corollary} 
\label{SO2n+1-transfer}
Let ${\rm Irr}^{\rm gen}_{\rm cusp}(G)$ be the set of isomorphic classes of
irreducible generic supercuspidal representations of $G$ and let ${\rm Irr}^{\rm isob}_{\rm cusp}({\rm GL}_{2r}(F))$
be the set of isomorphic classes of irreducible normalized induced representations of ${\rm GL}_{2r}(F)$ of the form ${\rm Ind}(\rho_1 \otimes \dotsm \otimes \rho_t)$ 
where each $\rho_i$ is an irreducible supercuspidal representation of ${\rm GL}_{2r_i}(F)$ such that $L(s,\rho_i,\wedge^2)$
has a pole at $s=0$ and $\rho_i \not\cong \rho_j$ for $i \neq j$. Then there is a unique injective map
\[
 \ell : {\rm Irr}^{\rm gen}_{\rm cusp}(G) \rightarrow {\rm Irr}^{\rm isob}_{\rm cusp}({\rm GL}_{2r}(F))
\]
satisfying
\begin{equation}
\label{SO-functorial-equality}
\gamma(s,\pi \times \tau, \psi)=\gamma(s,\ell(\pi) \times \tau,\psi)
\end{equation}
for any irreducible supercuspidal representation $\tau$ of ${\rm GL}_n(F)$ with $1 \leq n \leq r$.
\end{corollary}

\begin{proof}
The local Langlands lift satisfying \eqref{SO-functorial-equality} is given by Lomel\'{\i} \cite[Theorem 9.6]{Lom09}.
The injectivity of the map $\ell$ follows from Theorem \ref{SO-converse-supercusp}. The uniqueness of such a map $\ell$
is a consequence of the local converse theorem for ${\rm GL}_{2r}(F)$ established by Chai \cite{Cha19} and Jacquet--Liu \cite{JL18}.
\end{proof}

The local Langlands transfer $\ell$ to ${\rm GL}_{2r}(F)$ is expected to be surjective, and as a matter of fact this is true in the case of characteristic zero at least \cite[Proposition 6.1]{JS03}.

\section{The Local Converse Theorem for Sp$(2r)$}
\label{LCSp2r}

\subsection{Structure theory of symplectic groups and Weil representations}
For a positive integer $k$, we defined the element $J_k \in {\rm GL}_k(F)$ by
\[
J_k= \begin{pmatrix} &&1\vspace{-1ex} \\ &\vspace{-1ex}\iddots& \\ \vspace{-1ex}1 &&  \end{pmatrix}.
\]
We define
\[
 {\rm Sp}_{2r}(F)=\left\{ g \in {\rm GL}_{2r}(F) \, \middle|\,  \det g=1, {^tg}\begin{pmatrix} & J_r \\ -J_r &  \end{pmatrix}g=\begin{pmatrix} & J_r \\ -J_r &  \end{pmatrix} \right\}.
\]
We let 
\[
  w_{\ell}=\begin{pmatrix} & J_r \\ -J_r &  \end{pmatrix}
\]
represent the long Weyl element of ${\rm Sp}_{2r}(F)$. Let $B_{2r}=A_{2r}U_{2r}$ be the standard upper triangular Borel subgroup of $ {\rm Sp}_{2r}(F)$
with the maximal torus $A_{2r}$ and the standard maximal unipotent subgroup $U_{2r}$.

\par
We let $Q_{2n}$ be the Siegel parabolic subgroup, which has the Levi decomposition $Q_{2n}=M_{2n} \ltimes V_{2n}$ with 
\[
 M_{2n}=\left\{ {\textbf m}_n(a):= \begin{pmatrix} a & \\ & a^{\ast} \end{pmatrix}  \,\middle|\, a \in {\rm GL}_n(F) \right\} \cong {\rm GL}_n(F)
\]
and 
\[
 V_{2n}=\left\{ {\textbf u}_n(b):=\begin{pmatrix} I_n & b \\ & I_n  \end{pmatrix}  \in  {\rm Sp}_{2n}(F) \right\},
\]
where $a^{\ast}=J_n \prescript{t}{}{a}^{-1} J_n$. The condition ${\textbf u}_n(b) \in {\rm Sp}_{2n}(F)$ imposes the relation $J_nb=\prescript{t}{}{b}J_n$.
To alleviate notational burden, we simply write $B=B_{2r}$, $A=A_{2r}$, $U=U_{2r}$, and $G={\rm Sp}_{2n}(F)$.

\par
For $n < r$, we denote
\[
  w_{n,r-n}=\begin{pmatrix} &I_n&& \\ I_{r-n} &&& \\ &&&I_{r-n} \\ &&I_{n}& \end{pmatrix} \quad \text{and} \quad w_0=\begin{pmatrix} & I_n \\ -I_n & \end{pmatrix}.
\]
We embed ${\rm Sp}_{2n}(F)$ into $G$ via
\[
  j_{n}(g)=\begin{pmatrix} I_{r-n}&& \\ &g&\\ &&I_{r-n} \end{pmatrix}.
\]
and we identify an element of  ${\rm Sp}_{2n}(F)$ with an element of $G$ through the embedding $j$ without any further announcement.

\par
Let $\widetilde{\rm Sp}_{2n}$ be the metaplectic double cover of ${\rm Sp}_{2n}(F)$, realized using the normalized Rao cocycle \cite{Rao93}, which is a map $c : {\rm Sp}_{2n}(F) \times {\rm Sp}_{2n}(F) \rightarrow \{ \pm 1\}$. A prototype element of ${\rm Sp}_{2n}$ takes a form $(g,\epsilon) \in  {\rm Sp}_{2n}(F) \times \{ \pm 1\}$ with the multiplication 
\[
 (g_1,\epsilon_1)(g_2,\epsilon_2)=(g_1g_2,\epsilon_1\epsilon_2c(g_1,g_2)).
\]
An element $g \in {\rm Sp}_{2n}(F)$ can be viewed as $\textbf{s}(g) \in \widetilde{\rm Sp}_{2n}$ via $g \mapsto \textbf{s}(g):=(g,1)$.
In general, the map $g \mapsto \textbf{s}(g)$ is not a group homomorphism. For each set $Y \subset {\rm Sp}_{2n}(F)$, we let $\widetilde{Y}=pr^{-1}(Y)$
the metaplectic preimage of $Y$ under the canonical projection $pr:  \widetilde{\rm Sp}_{2n} \rightarrow  {\rm Sp}_{2n}(F)$ given by $pr (g,\epsilon)=g$
for $g \in  {\rm Sp}_{2n}(F)$ and $\epsilon \in \{ \pm 1 \}$.

\par
Let $F^{2n}$ be the $2n$-dimensional space of row vectors. We equip $F^{2n}$ with a symplectic form $\langle \cdot ,\cdot \rangle$ defined by $\langle v_1,v_2 \rangle=2v_1 w_{\ell} \prescript{t}{}{v_2}$. Let $\mathcal{H}(F^{2n})=F^{2n} \oplus F$ be the $(2n+1)$-dimensional Heisenberg group with the binary operation defined by
\[
  [v_1,t_1]\cdot  [v_2,t_2]=\left[ v_1+v_2, t_1+t_2+\frac{1}{2}\langle v_1,v_2 \rangle \right], \quad v_1, v_2 \in F^n, t_1, t_2 \in F.
\]
We define the following subgroups of  $\mathcal{H}(F^{2n})$:
\[
  X_n=\{ [(x,0),0] \,|\, x \in F^n\},\quad Y_n=\{ [(0,y),0] \,|\, y \in F^n\}, \quad Z_n=\{ [(0,0),z] \,|\, z \in F\}.
\]
Let 
\[
  \mathcal{H}_n=\left\{ \begin{pmatrix} 1 &x&y& z \\ &I_n&& J_n\prescript{t}{}{y} \\ &&I_n&-J_n\prescript{t}{}{x} \\ &&&1  \end{pmatrix} \,\middle|\, x,y \in F^n, z \in F \right\}.
\]
The group $\mathcal{H}_n$ can be viewed as the Heisenberg group $\mathcal{H}(F^{2n})$.
Let $\psi$ be a non-trivial additive character of $F$.
Let $\mathcal{S}(F^n)$ be the space of Schwartz–Bruhat functions on the row space $F^n$. We define the Fourier transform $\hat{\phi}$ of $\phi \in \mathcal{S}(F^n)$ by
\[
  \hat{\phi}(y)=\int_{F^n} \phi(x) \psi(2x J_n \prescript{t}{}{y}) \,dx.
\]
For $a \in F^{\times}$, let $\psi_a$ be the character of $F$ defined by $\psi_a(x)=\psi(ax)$.
Let $\gamma(\psi)$ be the Weil index of $x \mapsto \psi(x^2)$. We put
\[
  \gamma_{\psi}(a)=\frac{\gamma(\psi_a)}{\gamma(\psi)}.
\]
What is particularly important is that
\[
\gamma_{\psi}(1)=1, \quad \gamma_{\psi}(ab)=\gamma_{\psi}(a)\gamma_{\psi}(b)(a,b), \quad \gamma_{\psi}(a)=\gamma_{\psi}(a^{-1}), \quad \text{and} \quad \gamma_{\psi^{-1}}=\gamma^{-1}_{\psi}
\]
for $a, b \in F^{\times}$, where  $( \cdot ,\cdot )$ is the Hilbert symbol of $F$. Let $\omega_{\psi}$ be the Weil representation of $\mathcal{H}_n \rtimes \widetilde{\rm Sp}_{2n}$. It satisfies the following formul{\ae}:
\begin{itemize}
\item $\omega_{\psi}([(x,0),z])\phi(\xi)=\psi(z)\phi(\xi+x)$,
\item $\omega_{\psi}([(0,y),z])\phi(\xi)=\psi(2\xi J_n\prescript{t}{}{y})\phi(\xi)$,
\item $\omega_{\psi}({\textbf m}_n(a),\epsilon)=\epsilon \gamma_{\psi}(\det a)|\det a|^{1/2} \phi(\xi a)$,
\item $\omega_{\psi}({\textbf u}_n(b),\epsilon)\phi(\xi)=\epsilon \psi(\xi J_n\prescript{t}{}{b}\prescript{t}{}{\xi})\phi(\xi)$,
\item $\omega_{\psi}(w_0)\phi(\xi)=\beta_{\psi}\hat{\phi}(\xi)$,
\end{itemize}
where $\phi \in \mathcal{S}(F^n)$ and $\beta_{\psi}$ is a certain fixed eight root of unity \cite[p.404]{Kap15}.

\subsection{Uniqueness of Fourier--Jacobi models} Let $P=MN$ be the parabolic subgroup of $G$ with the Levi subgroup
\[
  M=\{ {\rm diag}(a_1,\dotsm,a_{r-n},g,a^{-1}_{r-n},\dotsm,a^{-1}_1) \,|\, a_i \in F^{\times}, g \in {\rm Sp}_{2n}(F) \}  \cong (F^{\times})^{r-n} \times {\rm Sp}_{2n}(F).
\]
We define the character on $N$ by
\[
  \psi_N^{-1}(u):=\psi^{-1} \left( \sum_{i=1}^{r-n-1} u_{i,i+1} \right) \;\; \text{for} \;\; u=(u_{i,j})_{1 \leq i,j \leq 2r} \in N.
\]
Let
 \[
 H=
 \begin{cases}
 j_{n}({\rm Sp}_{2n}(F)) \ltimes N & \text{for} \quad 1\leq n  < r, \\
  {\rm Sp}_{2r}(F) & \text{otherwise.}
 \end{cases}
 \]
In the matrix form, we find
\[
  H=\left\{ \begin{pmatrix} u& \ast & \ast \\ &g& \ast \\ && u^{\ast} \end{pmatrix} \,\middle|\, u \in U_{{\rm GL}_{r-n}(F)}, g \in {\rm Sp}_{2n}(F) \right\}.
\]
The group $\widetilde{\rm Sp}_{2n} \ltimes \mathcal{H}_n$ can be viewed as a subgroup of $\widetilde{H}=\widetilde{\rm Sp}_{2n} \ltimes N$ via 
the natural embedding $j_{n+1}$. We define the representation $\nu_{\overline{\psi}}$ of $\widetilde{H}$ by
\[
 \nu_{\overline{\psi}}(uh\widetilde{g})=\overline{\psi}_N(u) \overline{\omega}_{\overline{\psi}}(h \widetilde{g}) \quad \text{for}\quad u=\begin{pmatrix} z & v_1 & v_2 \\ &I_{2n+2}& v^{\ast}_1 \\  && z^{\ast}\end{pmatrix} \in N, h \in \mathcal{H}_n,  \widetilde{g} \in \widetilde{\rm Sp}_{2n}. 
\]
The pair $(\widetilde{H}, \nu_{\overline{\psi}})$ is called a {\it Fourier--Jacobi data} of ${\rm Sp}_{2r}(F)$. Since the group  $j_{n}({\rm Sp}_{2n}(F))$ is a quotient of $H$, a representation $\sigma$ of  $\widetilde{\rm Sp}_{2n}$ is deemed as a representation of $\widetilde{H}$, which allow us to construct the tensor product representation $\sigma \otimes \nu_{\overline{\psi}}$ of $\widetilde{H}$.

\begin{proposition}[Uniqueness of Fourier--Jacobi models]
 \label{F-J}
  Let $\pi$ be an irreducible smooth representation of ${\rm Sp}_{2r}(F)$ and $\sigma$
an irreducible genuine smooth representation of $\widetilde{\rm Sp}_{2n}$.  Then we have
 \[
  \dim {\rm Hom}_{H}(\pi,\sigma \otimes \nu_{\overline{\psi}}) \leq 1.
 \]
\end{proposition}

\begin{proof}
The equal rank case $(n=r)$ is dealt with very recently in \cite[Theorem 1.11]{Mez}. In \cite[Theorem 16.1]{GGP12}, Gan, Gross, and Prasad
established that the multiplicity one theorem for the general rank case $(n < r)$ can be deduced from the equal rank case $(n=r)$.
For the record, it is noteworthy that the result extends to finitely generated (smooth and generic, as always) representations. To this end, we use the theory of derivatives of Bernstein and Zelevinski \cite{BZ76}, as done in \cite[\S 11]{GP87}, \cite[Appendix]{GSR99} and \cite[Section 8]{Sou93} (cf. \cite[Section 3.1]{Kap13}, \cite[p.407]{Kap15}).
\end{proof}

The Frobenius reciprocity guarantees that  $\pi$ can be embedded in the space ${\rm Ind}^{{\widetilde{\rm Sp}}_{2r}}_H(\sigma \otimes \nu_{\overline{\psi}})$. We call this unique embedding, a {\it Fourier--Jacobi model} for $\pi$ corresponding to the data $(H,\sigma \otimes \nu_{\overline{\psi}})$.

\subsection{The local zeta integrals and gamma factors}
We let $n \leq r$ be a positive integer and We denote by $(\tau,V_{\tau})$ an irreducible $\psi^{-1}_{U_{{\rm GL}_{n}(F)}}$-generic representation of  ${\rm GL}_n(F)$.
For $s \in \mathbb{C}$, the genuine representation $\tau_s \otimes \gamma^{-1}_{\psi}$ of the double cover $\widetilde{M}_{2n}$ of $M_{2n}$ is defined by
\[
 \tau_s \otimes \gamma^{-1}_{\psi}(({\textbf m}_n(a),\epsilon))=\epsilon\gamma^{-1}_{\psi}(\det a)|\det a|^{s-1/2} \tau(a) \quad \text{for $a \in {\rm GL}_{n}(F)$ and $\epsilon \in \{ \pm 1 \}$.}
\]
We build the normalized induced representation $\widetilde{I}(s,\tau,\psi)={\rm Ind}^{{\widetilde{\rm Sp}}_{2n}}_{\widetilde{Q}_{2n}}(\tau_s\otimes \gamma^{-1}_{\psi})$.
The space  $\widetilde{I}(s,\tau,\psi)$ may be then viewed as the space of functions $\xi_s : {\widetilde{\rm Sp}}_{2n} \rightarrow V_{\tau}$ which transform in the following way
\[
 \xi_s(({\textbf m}_n(a),\epsilon){\textbf u}_n(b)\widetilde{g})=\epsilon \delta^{1/2}_{Q_{2n}}({\textbf m}_n(a))|\det a|^{s-1/2}\gamma^{-1}_{\psi}(\det a) \tau(a)\xi_s(\widetilde{g})
\]
for any $({\textbf m}_n(a),\epsilon) \in \widetilde{M}_{2n}$, ${\textbf u}_n(b) \in V_{2n}$, and $\widetilde{g} \in  {\widetilde{\rm Sp}}_{2n}$, where $\delta_{Q_{2n}}$ is a modulus 
character computed by $\delta_{Q_{2n}}(({\textbf m}_n(a),\epsilon))=|\det a|^{n+1}$ for $a \in {\rm GL}_n(F)$ and $\epsilon \in \{\pm 1 \}$ (see \cite[\S 4.1]{CST17} and \cite[Example 4.2-(1)]{Kim04}). We fix a nonzero Whittaker functional $\lambda \in {\rm Hom}_{U_{{\rm GL}_{n}(F)}}(\tau,\psi^{-1}_{U_{{\rm GL}_{n}(F)}})$. For $\xi_s \in \widetilde{I}(s,\tau,\psi)$, let $ f_{\xi_s} : {\widetilde{\rm Sp}}_{2n} \times {\rm GL}_n(F)$ be the $\mathbb{C}$-valued function defined by
\[
 f_{\xi_s}(\widetilde{g},a)=\lambda(\tau(a)\xi_s(\widetilde{g})) \quad \text{for $\widetilde{g} \in {\widetilde{\rm Sp}}_{2n}$ and $a \in  {\rm GL}_n(F)$.}
\]
Let $V^{{\widetilde{\rm Sp}}_{2n}}_{\widetilde{Q}_{2n}}(s,\mathcal{W}(\tau,\psi^{-1}_{U_{{\rm GL}_{n}(F)}}))$ be the space of functions $\{ f_{\xi_s} \,|\, \xi_s \in \widetilde{I}(s,\tau,\psi) \}$.
We define a (standard) intertwining operator 
\[
M(s,\tau,\psi^{-1}) : V^{{\widetilde{\rm Sp}}_{2n}}_{\widetilde{Q}_{2n}}(s,\mathcal{W}(\tau,\psi^{-1}_{U_{{\rm GL}_{n}(F)}})) \rightarrow V^{{\widetilde{\rm Sp}}_{2n}}_{\widetilde{Q}_{2n}}(1-s,\mathcal{W}(\tau^{\ast},\psi^{-1}_{U_{{\rm GL}_{n}(F)}}))
\]
 by
 \[
  M(s,\tau,\psi^{-1})f_s(\widetilde{g},a)=\int_{V_{2n}} f_s(\textbf{s}(w_0^{-1}v)\widetilde{g},d_na^{\ast}) \,dv, \quad \widetilde{g} \in \widetilde{\rm Sp}_{2n}(F), a \in {\rm GL}_n(F).
 \]

\par
Let $\psi_{U}$ be the generic character of $U$ defined by
\[
  \psi_{U}(u)=\psi \left( \sum_{i=1}^r u_{i,i+1} \right), \quad u=(u_{i,j})_{1 \leq i,j \leq 2r} \in U.
\]
For $W \in \mathcal{W}(\pi,\psi_U)$ $\phi \in \mathcal{S}(F^n)$, and $f_s \in V^{{\widetilde{\rm Sp}}_{2n}}_{\widetilde{Q}_{2n}}(s,\mathcal{W}(\tau,\psi^{-1}_{U_{{\rm GL}_{n}(F)}}))$ over non-archimedean local fields including positive characteristic of our primary interest, 
the {\it zeta} integrals 
\begin{multline*}
  \Psi(W,\phi,f_s)
  =\int_{U_{2n} \backslash {\rm Sp}_{2n}(F)} \int_{F^n}   \int_{{\rm Mat}_{r-n-1 \times n}(F)} \\  W \left( w_{n,r-n} {\textbf m}_r \begin{pmatrix} I_{r-n-1} &&y \\  &1&\\ &&I_n \end{pmatrix} j_{n+1} ( \begin{pmatrix} 1 &x&&  \\ &I_n&&  \\ &&I_n&-J_n\prescript{t}{}{x} \\ &&&1  \end{pmatrix} )j_{n}(g)   w^{-1}_{n,r-n} \right) \\
  \omega_{\overline{\psi}}(g) \phi(x)
  f_s(g,I_n) \,dy dx dg
\end{multline*}
are defined by Ginzburg, Rallis, and Soudry \cite[\S 3]{GRS98} for $n < r$, and the {\it zeta} integrals 
\[
  \Psi(W,\phi,f_s)=\int_{U_{2r} \backslash {\rm Sp}_{2r}(F)} W(g) \omega_{\overline{\psi}}(g)\phi(e_r)f_s(g,I_r) \,dg
\]
are defined by Gelbart and Piatetski-Shapiro \cite[\S 4]{GP87} for $n=r$. The integral $\Psi(W,\phi,f_s)$ converges absolutely for ${\rm Re}(s)$ sufficiently large and has meromorphic continuations to functions in $\mathbb{C}(q^{-s})$. Moreover the poles of $\Psi(W,\phi,f_s)$ belongs to a finite set which depends solely on the representation.
Since $\omega_{\overline{\psi}}$ and $f_s$ are genuine, their products $\widetilde{g} \mapsto  \omega_{\overline{\psi}}(\widetilde{g}) \phi(x)
  f_s(\widetilde{g},I_n)$ and $\widetilde{g} \mapsto \omega_{\overline{\psi}}(\widetilde{g})\phi(e_r)f_s(\widetilde{g},I_r)$ must factor through the natural map $\widetilde{\rm Sp}_{2n} \rightarrow {\rm Sp}_{2n}(F)$. We may therefore view functions $\omega_{\overline{\psi}}(\widetilde{g}) \phi(x)
  f_s(\widetilde{g},I_n)$ and $ \omega_{\overline{\psi}}(\widetilde{g})\phi(e_r)f_s(\widetilde{g},I_r)$ as a function on ${\rm Sp}_{2n}(F)$ and omit the section ${\textbf s}$.

\begin{proposition}[Local functional equation]
There exists a meromorphic function $\Gamma(s,\pi \times \tau,\psi)$ in $\mathbb{C}(q^{-s})$ such that
\[
 \Psi(W,\phi,M(s,\tau,\psi^{-1})f_s)=\gamma(s,\pi \times \tau,\psi) \Psi(W,\phi,f_s)
\]
for all $W \in \mathcal{W}(\pi,\psi_U)$, $\phi \in \mathcal{S}(F^n)$, and $f_s \in V^{{\widetilde{\rm Sp}}_{2n}}_{\widetilde{Q}_{2n}}(s,\mathcal{W}(\tau,\psi^{-1}_{U_{{\rm GL}_{n}(F)}}))$.
\end{proposition}

\begin{proof}
We denote by $\pi^{w_{n,r-n}}$ the representation $g \mapsto w_{n,r-n}gw^{-1}_{n,r-n}$ of $G$. 
In the domain of the convergence, the integrals $\Psi(W,\phi,f_s)$ and $ \Psi(W,\phi,M(s,\tau,\psi^{-1})f_s)$ may be regarded as trilinear forms in the space
\[
 {\rm Hom}_H(\pi^{w_{n,r-n}}\otimes V^{{\widetilde{\rm Sp}}_{2n}}_{\widetilde{Q}_{2n}}(s,\mathcal{W}(\tau,\psi^{-1}_{U_{{\rm GL}_{n}(F)}}))\otimes \nu_{\overline{\psi}},\mathbb{C}).
\]
Outside a discrete subset of $s$, $V^{{\widetilde{\rm Sp}}_{2n}}_{\widetilde{Q}_{2n}}(s,\mathcal{W}(\tau,\psi^{-1}_{U_{{\rm GL}_{n}(F)}}))$ becomes an induced irreducible representation,
for which uniqueness of Fourier--Jacobi models, Proposition \ref{F-J}, is applicable. Then local gamma factors $\gamma(s,\pi \times \tau,\psi)$ is defined as 
a proportionality constant between two zeta integrals $\Psi(W,\phi,f_s)$ and $ \Psi(W,\phi,M(s,\tau,\psi^{-1})f_s)$.
\end{proof}

The normalized intertwining operator is
\[
  N(s,\tau,\psi^{-1})=C^{\widetilde{\rm Sp}_{2n}}_{\psi}(s-1/2,\tau,w_0)M(s,\tau,\psi^{-1}),
\]
where $C^{\widetilde{\rm Sp}_{2n}}_{\psi}(s-1/2,\tau,w_0)$ is the {\it Langlands--Shahidi local coefficients} defined by Shahidi's functional equation \cite[(3.7)]{Kap15};
  \begin{multline*}
 \int_{V_{2n}} f_s({\textbf{s}(\textbf m}_n(d_n)w_0^{-1}v),I_n) \psi(v_{n,n+1}) \,dv\\
 = C^{\widetilde{\rm Sp}_{2n}}_{\psi}(s-1/2,\tau,w_0)\int_{V_{2n}} M(s,\tau,\psi^{-1})f_s ({\textbf m}_n(d_n)w_0^{-1}v,I_n) \psi(v_{n,n+1}) \,dv.
\end{multline*}
We denote by $\underset{\mathbb{C}[q^{\pm s}]^{\times}}{\sim}$ the equivalence relation of equality up to a monomial in $\mathbb{C}[q^{\pm s}]^{\times}$.

\begin{lemma}\cite[Theorem 4.2]{Szp13}
\label{metaplectic-gamma}
Let $\tau_0$ be an irreducible supercuspidal representation of ${\rm GL}_n(F)$. Then we have
\[
 C^{\widetilde{\rm Sp}_{2n}}_{\psi}(s,\tau_0,w_0) \underset{\mathbb{C}[q^{\pm s}]^{\times}}{\sim} \gamma(2s,\tau_0,\mathrm{Sym}^2,\psi),
\]
where $\gamma(s,\tau_0,{\rm Sym}^2)$ is a Langlands--Shahidi symmetric square $\gamma$-factor in
\citelist{ \cite{Lom15}*{\S 2.2} \cite{Lom16}*{\S 6.4 Theorem}}.
\end{lemma}

\begin{proof}
With an appropriate twist by unramified characters, we assume that $\tau_0$ is a unitary representation, and then globalize it to find a globally generic cuspidal automorphic representation $\otimes'_v \tau_v$ satisfying conditions in \cite[Lemma 3.1]{Lom17} (cf. \cite[Proposition 3.1]{Lom15}). 
The Shahidi's crude functional equations for the standard and symmetric factors \citelist{\cite{Lom09}*{Theorem 5.14} \cite{Lom16}*{Theorem 5.1}  \cite{Lom17}*{Theorem 4.3}}  (cf. \cite{Lom15}*{\S 2.3}) read
\[
 L^S(s,\tau)=\prod_{v \in S} \gamma(s,\tau_v,\psi_v) L^S(1-s,\tau^{\vee})
\]
and
\[
 L^S(s,\tau,{\rm Sym}^2)=\prod_{v \in S} \gamma(s,\tau_v,\psi_v,{\rm Sym}^2)L^S(1-s,\tau^{\vee},{\rm Sym}^2).
\]
Here, $L^S(s,\tau)$ and $L^S(s,\tau,{\rm Sym}^2)$ are partial $L$-function analogously defined as \eqref{partial-L-function} with respect to $S$, and
$\gamma(s,\tau_v,\psi_v)$ denote a Godement--Jacquet $\gamma$-factor as in \cite[\S 2.2]{Lom15}.
We rewrite the crude functional equation for metaplectic groups \cite[Theorem 4.1]{Szp13} 
\[
 \frac{L^S(2s,\tau,\mathrm{Sym}^2)}{L^S(s+1/2,\tau)}=\prod_{v \in S} C^{\widetilde{\rm Sp}_{2n}}_{\psi_v}(s,\tau_v,w_0) \frac{L^S(1-2s,\tau^{\vee},\mathrm{Sym}^2)}{L^S(-s+1/2,\tau^{\vee})}
\]
as
\[
 \prod_{v \in S}\frac{\gamma(2s,\tau_v,\mathrm{Sym}^2,\psi_v)}{\gamma(s+1/2,\tau_v,\psi_v)}=\prod_{v \in S} C^{\widetilde{\rm Sp}_{2n}}_{\psi_v}(s,\tau_v,w_0),
\]
where $\tau_v$ is the generic constituent of principal series representations for $v \in S-\{ v_0 \}$, appealing to \cite[Lemma 3.1, Remark 3.2]{Lom17}.
At the remaining places $v \in S-\{ v_0 \}$, \cite[Lemma 3.6]{Szp13} gives
\[
 C^{\widetilde{\rm Sp}_{2n}}_{\psi_v}(s,\tau_v,w_0) \underset{\mathbb{C}[q_{k_v}^{\pm s}]^{\times}}{\sim} \frac{\gamma(2s,\tau_v,\mathrm{Sym}^2,\psi_v)}{\gamma(s+1/2,\tau_v,\psi_v)}.
\]
The desired result is immediate from the fact that $\gamma(s+1/2,\tau_{v_0},\psi_{v_0})$ is invertible in $\mathbb{C}[q_{k_{v_0}}^{\pm s}]$ for any irreducible supercuspidal representation $\tau_{v_0}$. The proofs of  \cite[Theorem 4.1]{Szp13} and \cite[Lemma 3.6]{Szp13} are stated for number fields and non-archimedean local fields of characteristic zero, but the arguments work over
function fields and positive characteristic fields, respectively, without making any changes.
\end{proof}

In perspective of \cite{CST17}, the ambiguity of the equality up to $\mathbb{C}[q^{\pm s}]^{\times}$ in Lemma \ref{metaplectic-gamma} 
might be specified as the ``actual" equality. Apparently, the best thing we can do at this stage is to obtain the weak equality which is
enough for all the application coming after. Now, the local Rankin-Selberg gamma factor $\gamma(s,\pi \times \tau,\psi)$ is defined by the identity
\[
 \Psi(W,\phi,N(s,\tau,\psi^{-1})f_s)=\omega_{\pi}(-1)^n\omega_{\tau}(-1)^r\gamma(s,\pi \times \tau,\psi) \Psi(W,\phi,f_s).
\]
Consequently, we have the formula
\begin{equation}
\label{Sp2nUnnormal}
 \Gamma(s,\pi \times \tau,\psi) =\omega_{\pi}(-1)^n\omega_{\tau}(-1)^r\frac{\gamma(s,\pi \times \tau,\psi)}{C^{\widetilde{\rm Sp}_{2n}}_{\psi}(s-1/2,\tau,w_0) }
\end{equation}
which demonstrates that $\gamma(s,\pi \times \tau,\psi))$ is a rational function in $q^{-s}$. The rationale behind incorporating certain central characters $\omega_{\pi}(-1)^n\omega_{\tau}(-1)^r$
is to get the ``clean" multiplicativity formulas, Theorem \ref{SP2nproperty}-\ref{SP2nproperty-4}, similar to those satisfied by Shahidi's $\gamma$-factors.
If $r=0$, we set $G=\{1 \}$. We now turn toward concocting  a system of $\gamma$-factors on $G \times {\rm GL}_n(F)$.
 
\begin{theorem}
\label{SP2nproperty}
 Let $\pi$ and $\tau$ be a pair of irreducible admissible generic representations of ${\rm Sp}_{2r}(F)$ and ${\rm GL}_n(F)$.
The $\gamma$-factor $\gamma(s,\pi \times \tau,\psi)$ satisfies the following properties.
 \begin{enumerate}[label=$(\mathrm{\roman*})$]
\item {\rm (Naturality)} \label{SP2nproperty-1} Let $\eta : F' \rightarrow F$ be an isomorphism of local fields. Via $\eta$, $\pi$ and $\tau$ define 
irreducible admissible generic representations $\pi'$ of ${\rm Sp}_{2r}(F')$ and $\tau'$ of ${\rm GL}_n(F')$, and a nontrivial additive character $\psi'$ of $F'$.
Then we have
\[
 \gamma(s,\pi \times \tau,\psi)=\gamma(s,\pi' \times \tau',\psi').
\]
\item {\rm (Isomorphism)} \label{SP2nproperty-2}  If $\pi'$ and $\tau'$ are irreducible admissible generic representations of ${\rm Sp}_{2r}(F)$ and of ${\rm GL}_n(F)$ such that $\pi \cong \pi'$
and $\tau \cong \tau'$, then 
\[
 \gamma(s,\pi \times \tau,\psi)=\gamma(s,\pi' \times \tau',\psi).
\]
\item {\rm (Minimal cases)} \label{SP2nproperty-3}  For $r=0$, $\gamma(s,\pi \times \tau,\psi)=1$.
\item {\rm (Multiplicativity)}
\label{SP2nproperty-4} Let $\pi$ $($respectively, $\tau$$)$ be the irreducible generic quotient of a representation parabolically induced from
\[
 {\rm Ind}^{{\rm Sp}_{2r}(F)}_{\mathcal{P}_{2r}}(\pi_1 \otimes \pi_2 \otimes \dotsm \otimes \pi_t \otimes \pi_0)
\]
$($respectively $ {\rm Ind}^{{\rm GL}_{n}(F)}_{\mathcal{P}_n}(\tau_1 \otimes \tau_2 \otimes \dotsm \otimes \tau_s)$$)$, where  $\mathcal{P}_{2r}$ $($respectively $\mathcal{P}_n$$)$  is the standard  
parabolic subgroup with its Levi part isomorphic to $\prod_{i=1}^t {\rm GL}_{r_i}(F) \times {\rm Sp}_{2r_0}(F)$ 
$($respectively $\prod_{j=1}^s {\rm GL}_{n_j}(F)$$)$ associated to the partition $(r_1,\dotsm,r_t,r_0)$ of r $($respectively $(n_1,n_2,\dotsm,n_s)$ of $n$$)$. Then
\[
\gamma(s,\pi \times \tau,\psi)=\prod_{j=1}^s \gamma(s,\pi_0 \times \tau_j,\psi) \prod_{i=1}^t \prod_{j=1}^s \gamma(s,\pi_i \times \tau_j,\psi)  \gamma(s,\pi^{\vee}_i \times \tau_j,\psi).
\]
Here $\gamma(s,\pi_i \times \tau_j,\psi)$ is the ${\rm GL}_{r_i} \times {\rm GL}_{n_j}$ $\gamma$-factor defined by Jacquet, Piatetski-Shapiro and Shalika (cf. \cite{Cog14}).
When $t \geq 1$ and $r_0=0$, we make the following convention that we take $\pi_0$ to be the trivial character of ${\rm GL}_1(F)$ $($cf. \cite[\S 1.4-(iv)]{Lom15}$)$.
\item {\rm (Dependence of $\psi$)}
\label{SP2nproperty-5}
 For any $a \in F^{\times}$, let $\psi^a$ be the character given by $\psi^a(x)=\psi(ax)$. Then
\[
 \gamma(s,\pi \times \tau,\psi^a)=\omega_{\tau}(a)^{2r}|a|^{2nr(s-\frac{1}{2})}\gamma(s,\pi \times \tau,\psi).
\]
\item {\rm (Stability)} 
\label{SP2nproperty-6}
Let $\pi_1$ and $\pi_2$ be irreducible admissible generic representations of ${\rm Sp}_{2r}(F)$ with the same central character $\omega_{\pi_1}=\omega_{\pi_2}$.
For $\chi$ sufficiently highly ramified, we have
\[
 \gamma(s,\pi_1 \times \chi,\psi)= \gamma(s,\pi_2 \times \chi,\psi).
\]
\item {\rm (Global functional equation)} 
\label{SP2nproperty-7}
We assume that $\pi$ and $\tau$ are globally generic irreducible cuspidal automorphic representations of ${\rm Sp}_{2r}(\mathbb{A}_k)$ and ${\rm GL}_n(\mathbb{A}_k)$.
Let $S$ be a finite set of places such that for all $v \notin S$, all data are unramified. Then
\[
  L^S(s,\pi \times \tau)=\prod_{v \in S} \gamma(s,\pi_v \times \tau_v,\psi_v) L^S(1-s,\pi^{\vee} \times \tau^{\vee}),
\]
where $L^S(s,\pi \times \tau)$ is the partial $L$-function with respect to $S$.
\end{enumerate}
\end{theorem}

\begin{proof}
The properties \ref{SP2nproperty-1}, \ref{SP2nproperty-2}, and \ref{SP2nproperty-3} are immediate. As in \cite{Kap13}, the hard part should be multiplicativity \ref{SP2nproperty-4}, which follows from Kaplan \cite[\S 8-\S 9]{Kap15}. The dependence of $\gamma$-factors \ref{SP2nproperty-5} is explained in \cite[\S 7]{Kap15}. The proof of stability of $\gamma$-factors \ref{SP2nproperty-6} in \cite{Zha17-2} is purely local, and go through in the case of local function fields.
Concerning Statement \ref{SP2nproperty-7}, we take a direct path, which bypass the use of the involution of M{\oe}glin, Vign\'eras, and Waldspurger as commented in \cite[\S 5]{Kap13}.
Let $\pi_v={\rm Ind}^G_{B}(\mu)$ and $\tau_v={\rm Ind}^{{\rm GL}_n(F)}_{B_{n}}(\chi)$ be irreducible generic unramified representations of $G$ and ${\rm GL}_n(F)$,
where $\mu$ and $\chi$ are unramified characters of $B$ and $B_{n}$. If
\[
 t={\rm diag}(a_1,\dotsm,a_r,a^{-1}_r,\dotsm,a^{-1}_1) \in A,
 \] 
 then $\mu(t)=\mu_1(t_1)\mu_2(t_2) \dotsm\mu_r(t_r)$. Let
 \[
   A_{\pi_v}={\rm diag}(\mu_1(\varpi),\dotsm,\mu_r(\varpi), 1, \mu^{-1}_r(\varpi),\dotsm,\mu^{-1}_1(\varpi)) \in {\rm SO}_{2r+1}(\mathbb{C})
 \]
 be the semisimple conjugacy class of ${\rm SO}_{2r+1}(\mathbb{C})$ associated with $\pi_v$. 
 Since the Satake parameters $A_{\pi_v}$ and $A_{\pi_v^{\vee}}$ are conjugate in the dual group ${\rm SO}_{2r+1}(\mathbb{C})$, we know that 
 $L(s,\pi_v \times \tau_v)=L(s, \pi_v^{\vee} \times \tau_v)$. To be more precisely, we attach to $\tau_v$ a semisimple conjugacy class in ${\rm GL}_n(\mathbb{C})$
 denoted by
 \[
  B_{\tau_v}={\rm diag}(\chi_1(\varpi),\dotsm,\chi_n(\varpi)).
 \]
 Then $L$-functions for unramified principal series representations are defined by
\begin{multline*}
   L(s,\pi_v \times \tau_v)=\frac{1}{\det(I_{2rn}-(A_{\pi_v} \otimes B_{\tau_v})q^{-s})}\\
   =\prod_{1 \leq k \leq n}  \frac{1}{1-\chi_k(\varpi)q^{-s}} \prod_{1 \leq i \leq r,1\leq j \leq n} \left( \frac{1}{1-\mu_i(\varpi)\chi_j(\varpi)q^{-s}} \cdot
   \frac{1}{1-\mu^{-1}_i(\varpi)\chi_j(\varpi)q^{-s}} \right).
\end{multline*}
 However, the contragredient representation $\pi_v^{\vee}$ is ${\rm Ind}^{G}_{B}(\mu^{-1})$ (cf. \cite[Chapter 4-\S 4 Fact (1)]{Kim04}) from which 
 we conclude that
  \[
   L(s,\pi_v^{\vee} \times \tau_v)=\frac{1}{\det(I_{(2r+1)n}-(A_{\pi_v^{\vee}} \otimes B_{\tau_v})q^{-s})}=\frac{1}{\det(I_{(2r+1)n}-(A^{-1}_{\pi_v} \otimes B_{\tau_v})q^{-s})} =L(s,\pi_v \times \tau_v). 
 \] 
Aided by unramified calculation of Rankin--Selberg by Ginzburg, Rallis, and Soudry \cite[\S 3.1]{GRS98}, Kaplan \cite[Theorem 1-(9)]{Kap15} deduces the following crude functional equation
from properties of Eisenstein series;
\[
 L^S(s,\pi \times \tau)=\prod_{v \in S} \gamma(s,\pi_v \times \tau_v,\psi_v)L^S(1-s,\pi \times \tau^{\vee}).
\]
We apply the equality $L(s,\pi_v\times\tau_v)=L(s,\pi_v^{\vee}\times \tau_v)$ to $\tau_v^{\vee}$ for each $v \notin S$ to get
\[
 L^S(1-s,\pi \times \tau^{\vee})=L^S(1-s,\pi^{\vee} \times \tau^{\vee}). \qedhere
\]
\end{proof}

Via \cite[Theorem 1.5]{Lom15}, a list of local properties in combination with their role in the global functional equation apparently characterize $\gamma$-factors on $G \times {\rm GL}_n(F)$.
The key ingredient to prove Theorem \ref{SP2n-equality} right below is the local to global result of Henniart--Lomel\'{\i} \cite[Proposition 3.1]{Lom15}.

\begin{theorem}
\label{SP2n-equality} 
The Rankin--Selberg $\gamma$-factor $\gamma(s,\pi \times \tau,\psi)$ agrees with Shahidi's $\gamma$-factor for $\pi \times \tau$ defined in \cite{Lom15}.
\end{theorem}

 \subsection{The local converse theorem}
 \label{Sp2n-localconverse}
 
 Many of results in \S \ref{Sp2n-localconverse} resemble those of \S \ref{SO2n+1-localconverse}, henceforth we only remark on the nature of differences or omit the proof most of the time.
 Our primary task is to formulate the local converse theorem for irreducible generic supercuspidal representations of $G$. Q. Zhang \cite{Zha18} proved the local converse theorem for ${\rm Sp}_{2r}(F)$
 using twisting by all generic representations of ${\rm GL}_k(F)$ with $1 \leq k \leq r$. It turns out that the proof of Theorem  \ref{Sp2r-supercusp} is reduced to the main theorem of  \cite{Zha18}.
 Consequently we only need need to twist by all generic supercuspidal representations of ${\rm GL}_k(F)$ in Theorem \ref{Sp2r-supercusp}.

 \begin{theorem}
 \label{Sp2r-supercusp}
Let $\pi$ and $\pi'$ be irreducible generic supercuspidal representations of $G$ with the common central character $\omega$. If
 \[
 \gamma(s,\pi \times \tau,\psi)=\gamma(s,\pi' \times \tau,\psi)
 \]
 for all irreducible generic supercuspidal representations $\tau$ of ${\rm GL}_k(F)$ and for all $k$ with $1 \leq k \leq r$, then $\pi \cong \pi'$. 
\end{theorem}


 We shall initiate the analysis for detecting the place of poles and zeros of twisted $\gamma$-factors to decide the concrete structure of the supercuspidal support for irreducible generic representations of $G$. Then the general case of the local converse theorem follows from the supercuspidal case. We need the following well-known Rallis Lemma concerning the pole of intertwining operators.
 
  \begin{lemma}$(${\rm Rallis Lemma, \cite[Lemma 4.11]{Tak15}}$)$
  \label{Rallis Lemma}
 Let $\tau$ be a character of $F^{\times}$. The highest pole of the intertwining operator
 $M(s,\tau,\psi^{-1})$ is achieved by 
 \[
 M(s,\tau,\psi^{-1}) f_s(\textbf{s}(w_0),1)
 \]
 as the sections $f_s$ are taken over those sections with ${\rm Supp}(f_s) \subseteq \widetilde{B}w_0\widetilde{B}$.
 \end{lemma}
The proof of Lemma \ref{metaplectic-gamma} explains that can be subdivided in accordance with the following cases;
\[
C^{\widetilde{\rm Sp}_{2n}}_{\psi}(s-1/2,\tau,w_0) \underset{\mathbb{C}[q^{\pm s}]^{\times}}{\sim}
 \begin{cases}
&\hspace{-.2cm} \dfrac{\gamma(2s-1,\tau,\mathrm{Sym}^2,\psi)}{\gamma(s,\tau,\psi)}, \quad \text{if $\tau$ is a character of $F^{\times}$;} \\
 &\hspace{-.2cm} \gamma(2s-1,\tau,\mathrm{Sym}^2,\psi), \quad \text{otherwise.} \\
 \end{cases}
\]
In this regard, we handle $n=1$ case independently. Hereafter until the end of this section, $L(s,\tau,{\rm Sym}^2)$ is the $L$-factor defined by Langlands--Shahidi pertaining to the symmetric square representation \cite{Lom09,Lom15,Lom16,Lom17}.

\begin{lemma}[Casselman--Shahidi Lemma]
\label{CS-SP2n}
Let $\tau$ be either an irreducible unitary supercuspidal representation of ${\rm GL}_n(F)$ with $n>1$ or a character of $F^{\times}$. The intertwining operator
\[
  \frac{1}{L(2s-1,\tau,{\rm Sym}^2)} M(s,\tau,\psi^{-1})
\]
is entire as a function of $s$.
\end{lemma}

\begin{proof}
We explore the case of irreducible unitary supercuspidal representations $\tau$ with $n > 1$.
Counting on \cite[Proposition 5.4]{Lom17} (cf. \cite[Corollary 7.6]{Sha90}) along with \cite[Lemma 6.1 and Theorem E]{Szp13},
it can be shown in the same manner as the proof of Proposition \ref{CS-SO2n} that $s=s_0$ is a pole of $M(s,\tau,\psi^{-1})$ if and only if $s=s_0$ contributes the pole of $L(2s-1,\tau,{\rm Sym}^2)$.

\par
We now assume that $\tau$ is a character of $F^{\times}$. 
Keeping Lemma \ref{Rallis Lemma} in mind, it is sufficient to show that
\[
\frac{1}{L(2s-1,\tau,{\rm Sym}^2)} M(s,\tau,\psi^{-1}) f_s(\textbf{s}(w_0),1) 
\]
is holomorphic with the function $ f_s$ supported on the big cell, ${\rm Supp}(f_s) \subseteq \widetilde{B}w_0\widetilde{B}$. 
For our convenience, we write $f_s(\tilde{g})$ for $f_s(\tilde{g},1)$. We have the Bruhat decomposition:
\[
  \textbf{s}( w^{-1}_0 {\textbf u}_1(b))\textbf{s}(w_0)=\textbf{s} \left( \begin{pmatrix} b^{-1} & 1 \\ & b \end{pmatrix} \right) \textbf{s}\left(w_0 \begin{pmatrix} 1 & -b^{-1} \\ & 1 \end{pmatrix}\right).
\]
It is extremely tedious to determin the precise factor $\epsilon$, though it is not so deep. Since we will not need the accurate information on $\epsilon$, we will leave $\epsilon$ for the rest of our computation. We can write
\begin{multline*}
  \int_F f_s(\textbf{s}( {\textbf m}_2(-1)w^{-1}_0 {\textbf u}_1(b))\textbf{s}(w_0))\, db= \int_F f_s\left(\textbf{s} \left( \begin{pmatrix} b^{-1} & 1 \\ & b \end{pmatrix} \right) \textbf{s}\left(w^{-1}_0 \begin{pmatrix} 1 & -b^{-1} \\ & 1 \end{pmatrix}\right) \right) \, db \\
  = \int_F \delta^{1/2}_B \left( \begin{pmatrix}b^{-1}& \\ & b \end{pmatrix} \right) |b^{-1}|^{s-1/2} \gamma^{-1}_{\psi}(b) \tau(b^{-1}) f_s\left( \textbf{s}\left(w^{-1}_0 \begin{pmatrix} 1 & -b^{-1} \\ & 1 \end{pmatrix} \right)\right) \, db.
\end{multline*}
Upon changing the additive measure $db$ to the multiplicative measure $d^{\times}b=db/|b|$, and then changing the variables $b^{-1} \mapsto b$,
the above integral can be expressed as
\begin{multline}
\label{coset-sum}
 \int_{F^{\times}} \epsilon \delta^{1/2}_B \left(\begin{pmatrix}b& \\ & b^{-1} \end{pmatrix} \right) |b|^{s-1/2} \gamma^{-1}_{\psi}(b) \tau(b) f_s\left( \textbf{s}\left(w^{-1}_0 \begin{pmatrix} 1 & -b \\ & 1 \end{pmatrix} \right) \right)  |b^{-1}|\, d^{\times}b\\
 = \int_{F^{\times}} \epsilon  |b|^{s-1/2} \gamma^{-1}_{\psi}(b) \tau(b) f_s\left( \textbf{s}\left(w^{-1}_0 \begin{pmatrix} 1 & -b \\ & 1 \end{pmatrix}\right) \right) \,d^{\times}b
\end{multline}
for some $\epsilon \in \{\pm 1 \}$. Let $b_1,b_2,\dotsm,b_{\ell}$ be a complete set of representatives ${F^{\times}}^2 \backslash F^{\times}$. In this way, we find that
the right hand side of \eqref{coset-sum} equals to
\[
 \sum_{i=1}^{\ell} \int_{{F^{\times}}^2} \epsilon_i |x|^{s-1/2} \tau(x)  \delta^{-1/2}_B \left(\begin{pmatrix}b_i& \\ & b^{-1}_i \end{pmatrix} \right)
 f_s\left( \textbf{s}\left(\begin{pmatrix} b_i & \\ & b^{-1}_i  \end{pmatrix}\right)\textbf{s}\left(w^{-1}_0 \begin{pmatrix} 1 & -xb_i \\ & 1 \end{pmatrix}\right) \right) \,d^{\times}x,
\]
for some $\epsilon_i \in \{ \pm 1 \}$. Owing to  \cite[Lemma 4.12]{Tak15}, each summand can be converted into
\[
\epsilon_i  \delta^{-1/2}_B \left(\begin{pmatrix}b_i& \\ & b_i^{-1} \end{pmatrix} \right)
 f_s\left( \textbf{s}\left(\begin{pmatrix} b_i & \\ & b^{-1}_i  \end{pmatrix})\textbf{s}(w^{-1}_0 \begin{pmatrix} 1 & -xb_i \\ & 1 \end{pmatrix}\right)\right)=\sum_{\lambda,\phi} \lambda(s)\phi(x)
\]
for some holomorphic functions $\lambda$ and smooth compactly supported function $\phi$. Therefore, the study of the analytic behavior of ensuing integrals boils down to the study of that of
\[
 \int_{{F^{\times}}^2}  |x|^{s-1/2} \tau(x) \phi(x) d^{\times} x=c \int_{F^{\times}} |y|^{2s-1} \tau^2(y) \phi(y) d^{\times} y
\]
for an appropriate non-zero constant $c$. But the last integral is precisely an entire function on $s$ times Tate's $L$-factor, $L(2s-1,\tau^2)=L(2s-1,\tau,\rm{Sym}^2)$.
\end{proof}


We state and prove the following corollary, which will be a part of Proposition \ref{PoleDecompositionSympletic} we look into shortly.

\begin{corollary}
\label{sp2n-zerospoles}
 Let $\pi_0$ be an irreducible supercuspidal representation of ${\rm Sp}_{2r}(F)$ and $\tau$ an irreducible unitary supercuspidal representation of ${\rm GL}_n(F)$. 
If the twisted gamma factor $\gamma(s,\pi_0\times \tau,\psi)$ has a pole at $s=1$, then the pole must be simple and the representation $\tau$ must be self-dual.
Moreover, $L(s,\tau,{\rm Sym}^2)$ has a pole at $s=0$. 
\end{corollary}

\begin{proof}
In perspective of definition of $\gamma(s,\pi_0 \times \tau,\psi)$, we have
\[
\begin{split}
\gamma(s,\pi_0 \times \tau,\psi)\Psi(W,f_s) &\underset{\mathbb{C}[q^{\pm s}]^{\times}}{\sim} \dfrac{\gamma(2s-1,\tau,\mathrm{Sym}^2,\psi)}{\gamma(s,\tau,\psi)}\Psi(W,M(s,\tau,\psi^{-1})f_s)\\
&\underset{\mathbb{C}[q^{\pm s}]^{\times}}{\sim} L(s,\tau)  \cdot  \frac{L(2(1-s),\tau,{\rm Sym}^2)}{L(1-s,\tau)}  \cdot \frac{\Psi(W,M(s,\tau,\psi^{-1})f_s)}{L(2s-1,\tau,{\rm Sym}^2)}.
\end{split} 
\]
Non-vanishing results \S 3.2 and Proposition 6.6 of \cite{GRS98} assure that
there exist a Whittaker function $W \in \mathcal{W}(\pi_0,\psi_{U})$ and a function $f_s\in  V^{\widetilde{\rm Sp}_{2n}}_{Q_{2n}}(s,\mathcal{W}(\tau,\psi^{-1}))$
such that $\Psi(W,f_s)$ is a non-zero constant.
Appealing to Proposition \ref{CS-SP2n}, the occurrence of poles of $\gamma(s,\pi_0 \times \tau,\psi)$ at $s=1$ is amount to saying the appearance of the pole of $L(s,\tau,\rm{Sym}^2)$ at $s=0$. Lastly, the factorization $L(s,\tau,\wedge^2)L(s,\tau,{\rm Sym}^2)=L(s,\tau \times \tau)$ \cite[Lemma 7.12]{Lom09} manifests that $\tau \cong \tau^{\vee}$, as required.
\end{proof}

There are various types of Jacquet's subqotient theorem which can be founded in the current literature such as \cite[Proposition 3.19, p.34]{BZ76} for the ${\rm GL}_n$ case and \cite[\S 2.1 Claim, p.55]{Ber92} for the connective reductive case.  Our case is the positive characteristic analogue of \cite[p.1111]{Liu11} and we summarize it in the following manner.

\begin{lemma}\cite[\S 3]{Liu11}
\label{SPcuspsupport}
Let $\pi$ be an irreducible admissible generic representation of ${\rm Sp}_{2r}(F)$. Then $\pi$ is a subquotient of 
the unitary induced representation 
\[
  {\rm Ind}^{{\rm Sp}_{2r}(F)}_{\mathcal{P}_{2r}}(\rho_1|\det|^{z_1} \otimes \dotsm \otimes \rho_t|\det|^{z_t}\otimes \pi_0),
\]
where $\mathcal{P}_{2r}$ is the standard parabolic subgroup with its Levi part isomorphic to $\prod_{i=1}^t {\rm GL}_{r_i}(F) \times {\rm Sp}_{2r_0}(F)$
associated to the partition $(r_1,\dotsm,r_t,r_0)$ of $r$, each $\rho_i$ is an irreducible unitary supercuspidal representation of ${\rm GL}_{r_i}(F)$
with the $z_i$ real and ordered so that $z_1 \geq z_2 \geq \dotsm \geq z_t \geq 0$, and $\pi_0$ is an irreducible generic supercuspidal representation of $ {\rm Sp}_{2r_0}(F)$.
\end{lemma}

We collect the fairly standard results from \cite[Proposition 2.9]{DNS15} and Corollary \ref{sp2n-zerospoles} that will be required soon after.

\begin{proposition} 
\label{PoleDecompositionSympletic}
Let $\tau$ be an irreducible unitary supercuspidal representation of ${\rm GL}_n(F)$. With the notation as in Lemma \ref{SPcuspsupport}, we have the followings;
\begin{enumerate}[label=$(\mathrm{\roman*})$]
\item\label{PoleDecompositionSympletic-1} If the product $\prod_{i=1}^t \gamma(s+z_i,\rho_i \times \tau,\psi)$ has a real pole $($respectively, a real zero$)$ at $s=s_0$, then $\tau \cong \rho_i^{\vee}$
and $s_0=1-z_i$ $($respectively, $s_0=-z_i)$ for some $1 \leq i \leq t$.
\item\label{PoleDecompositionSympletic-2} If the product $\prod_{i=1}^t \gamma(s-z_i,\rho_i^{\vee} \times \tau,\psi)$ has a real pole $($respectively, a real zero$)$ at $s=s_0$, then $\tau \cong \rho_i$
and $s_0=1+z_i$ $($respectively, $s_0=z_i)$ for some $1 \leq i \leq t$.
\item\label{PoleDecompositionSympletic-3} The factor $\gamma(s,\pi_0 \times \tau,\psi)$ has no zero for ${\rm Re}(s) > 0$. 
 If a real pole of $\gamma(s,\pi_0 \times \tau,\psi)$ occurs at $s=s_0$, then the pole must be a simple pole at $s_0=1$ and $\tau \cong \tau^{\vee}$. 
\end{enumerate}
\end{proposition}

\begin{proof}
The properties \ref{PoleDecompositionSympletic-1} and \ref{PoleDecompositionSympletic-2} are explained in \cite[Proposition 2.9]{DNS15}. In view of Theorem \ref{SP2n-equality}, the first part of the statement \ref{PoleDecompositionSympletic-3} follows from the holomorphy of tempered $L$-functions via Langlands--Shahidi methods \cite[\S 4.1-(ix)]{Lom15}. The second part of \ref{PoleDecompositionSympletic-3} can be otained from Corollary \ref{sp2n-zerospoles}.
\end{proof}

The following proposition seems to generalize results of \cite[Lemma 3.3 and Theorem 3.5]{Liu11}.

\begin{proposition}
 With the notation in Lemma \ref{SPcuspsupport}, suppose that $\rho'_i$ are irreducible unitary supercuspidal representations of ${\rm GL}_{r_i}(F)$ for $1 \leq i \leq t'$
 and $\pi'_0$ is an irreducible generic supercuspidal representation of ${\rm Sp}_{2r_0}(F)$ with $r=r'_0+\sum_{i=1}^{t'}r'_i$ and that $z'_1 \geq z'_2 \geq \dotsm \geq z'_t \geq 0$
 are real numbers. Suppose also that
\begin{multline*}
\left[\prod_{i=1}^t \gamma(s+z_i,\rho_i \times \tau,\psi)\gamma(s-z_i,\rho_i^{\vee} \times \tau,\psi) \right]\gamma(s,\pi_0 \times \tau,\psi)\\
 =\left[\prod_{i=1}^{t'} \gamma(s+z'_i,\rho'_i \times \tau,\psi)\gamma(s-z'_i,{\rho'_i}^{\vee} \times \tau,\psi) \right]\gamma(s,\pi'_0 \times \tau,\psi)
\end{multline*}
for all irreducible unitary supercuspidal representations $\tau$ of ${\rm GL}_n(F)$ with $1 \leq n \leq r$. Then $t=t'$ and there exists a permutation $\sigma$
of $\{ 1,2,\dotsm,t\}$ such that
\begin{enumerate}[label=$(\mathrm{\roman*})$]
\item $r_i=r'_{\sigma(i)}$ for all $i=1,2,\dotsm,t$;
\item $z_i=z'_{\sigma(i)}$ and $\rho_i \cong \rho_{\sigma(i)}$ for all $i=1,2,\dotsm,t$;
\item $\gamma(s,\pi_0 \times \tau,\psi)=\gamma(s,\pi'_0 \times \tau,\psi)$ for all irreducible unitary supercuspidal representations $\tau$ of ${\rm GL}_n(F)$ with $1 \leq n \leq r$. 
\end{enumerate}
\end{proposition}

We are finally ready to reveal the main result of this section. The compatible result can be carried over to $p$-adic fields in a parallel way.

\begin{theorem}
Let $\pi$ and $\pi'$ be irreducible admissible generic representations of $G$ with the common central character $\omega$. If
 \[
 \gamma(s,\pi \times \tau,\psi)=\gamma(s,\pi' \times \tau,\psi)
 \]
 for all irreducible generic supercuspidal representations $\tau$ of ${\rm GL}_n(F)$ and for all $n$ with $1 \leq n \leq r$, then $\pi \cong \pi'$. 
\end{theorem}

Having Theorem \ref{SP2n-equality} in hand, we construct the injectivity of the local Langlands lift or transfer. 

 \begin{corollary}
 \label{SP2n-transfer}
  Let ${\rm Irr}^{\rm gen}_{\rm cusp}(G)$ be the set of isomorphic classes of irreducible generic supercuspidal representations of $G$
 and let ${\rm Irr}^{\rm isob}_{\rm cusp}({\rm GL}_{2r+1}(F))$ be the set of isomorphic classes of irreducible normalized induced representations
 of ${\rm GL}_{2r+1}(F)$ of the form ${\rm Ind}(\rho_1 \otimes \dotsm \otimes \rho_t)$ where each $\rho_i$ is an irreducible supercuspidal representation of
 ${\rm GL}_{r_i}(F)$ such that $L(s,\rho_i,{\rm Sym}^2)$ has a pole at $s=0$ and $\rho_i \not\cong \rho_j$ for $i \neq j$. Then there is a unique injective map
 \[
  \ell : {\rm Irr}^{\rm gen}_{\rm cusp}(G) \rightarrow {\rm Irr}^{\rm isob}_{\rm cusp}({\rm GL}_{2r+1}(F))
 \]
 satisfying
 \[
  \gamma(s,\pi \times \tau,\psi)=\gamma(s,\ell(\pi)\times\tau,\psi)
 \]
 for any irreducible supercuspidal representation $\tau$ of ${\rm GL}_n(F)$ with $1 \leq n \leq r$.
 \end{corollary}

We again expect that the local Langlands transfer $\ell$ to ${\rm GL}_{2r+1}(F)$ is surjective.


\begin{acknowledgements}
The author is indebted to Q. Zhang for many stimulus discussions throughout this project. 
I am very grateful to M. Krishnamurthy for explain the genesis of the Langalands--Shahidi method countless times.
I would like to convey my deepest gratitude to J. Cogdell and F. Shahidi for incredible suggestions to apply
Cogdell, Shahidi, and Tsai's theory \cite{CST17} to other situation. Without their instruction, this article never have seen the light of the world.
I also thank to every members in Algebra and Number Theory group at the University of Maine, especially, J. Buttcane, B. Hanson, and A. Knightly for providing vibrant environment and their constant encouragement, while the paper was made. The author would like to convey our sincere appreciation to anonymous referee
for reading our paper and invaluable comments which significantly improve the exposition of this manuscript.
\end{acknowledgements}

\begin{funding statement}
This work was supported by the Ewha Womans University Research Grant of 2022 and
 the National Research Foundation of Korea (NRF) grant 
funded by the Korea government (No. RS-2023-00209992). 
\end{funding statement}

 \bibliographystyle{amsplain}

\end{document}